\documentclass[12pt]{amsart}
\usepackage{amssymb}
\usepackage{colordvi}
\usepackage{graphicx}
\usepackage{url} 
\usepackage{vmargin}
\setpapersize{USletter}
\setmargrb{2.5cm}{2.5cm}{2.5cm}{2.5cm} 

\numberwithin{equation}{section}
\newtheorem*{VC}{Vakil's Criterion}
\newtheorem*{conjecture}{Conjecture}
\newtheorem{theorem}{Theorem}
\newtheorem{proposition}[theorem]{Proposition}

\newtheorem{lemma}[theorem]{Lemma}

\theoremstyle{remark}

\newtheorem{remark}[theorem]{Remark}
\newtheorem{definition}[theorem]{Definition}

\newcounter{FNC}[page]
\def\fauxfootnote#1{{\addtocounter{FNC}{2}\Magenta{$^\fnsymbol{FNC}$}%
     \let\thefootnote\relax\footnotetext{\Magenta{$^\fnsymbol{FNC}$#1}}}}

\newcommand{\G}{{\mathbb G}}
\newcommand{\K}{{\mathbb K}}
\renewcommand{\P}{{\mathbb P}}

\newcommand{\adot}{a_\bullet}

\newcommand{\Gal}{{\mathcal G}}
\newcommand{\calK}{{\mathcal K}}
\newcommand{\calL}{{\mathcal L}}
\newcommand{\calS}{{\mathcal S}}

\newcommand{\PGL}{{\rm PGL}}

\newcommand{\pr}{\mbox{\it pr}}

\newcommand{\lhra}{\ensuremath{\lhook\joinrel\relbar\joinrel\relbar\joinrel\rightarrow}}

\def\Color#1#2{#2}
\def\Gre#1{\Color{1 0 1 0.3}{#1}}
\newcommand{\defcolor}[1]{\RoyalBlue{#1}}
\newcommand{\demph}[1]{\defcolor{{\sl #1}}}

\newcommand{\Kosfi}[8]{\begin{picture}(58,25)(-2.9,-2.5)
  \put(-3.4,-3){\includegraphics[height=24pt]{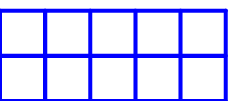}}
  \put( 0,11.5){{\small#1}}\put(11.75,11.5){\small#2}
  \put(23.5,11.5){\small#3}\put(35.25,11.5){\small#4}
  \put(47,11.5){\small#5}
  \put( 0,-0.3){\small#6}\put(11.75,-0.3){\small#7}
  \put(23.5,-0.3){\small#8}\put(35.25,-0.3){\small#8}
  \put(47,-0.3){\small#8}
 \end{picture}}
\def\Kosfo#1#2#3#4#5#6#7{\begin{picture}(47,25)(-2.9,-2.5)
  \put(-3.4,-3){\includegraphics[height=24pt]{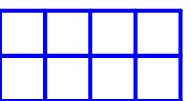}} 
  \put( 0,11.5){\small#1}\put(11.75,11.5){\small#2}
  \put(23.5,11.5){\small#3}\put(35.25,11.5){\small#4}
  \put( 0,-0.3){\small#5}\put(11.75,-0.3){\small#6}
  \put(23.5,-0.3){\small#7}\put(35.25,-0.3){\small#7}
 \end{picture}}
\title[Galois groups of Schubert problems of lines]{Galois groups of Schubert problems of lines
  are at least alternating} 
\author[Brooks]{Christopher J.\ Brooks}
\address{Department of Mathematics\\
         University of Utah\\
         Salt Lake City\\
         Utah \ 84112-0090\\
         USA
         }
\email{cbrooks@math.utah.edu}
\urladdr{\url{http://www.math.utah.edu/~cbrooks/}}
\author[Mart\'in del Campo]{Abraham Mart\'in del Campo}
\address{Abraham Mart\'in del Campo\\
         IST Austria\\
         Am Campus 1\\
         3400 Klosterneuburg\\
         Austria
         }
\email{abraham.mc@ist.ac.at}
\urladdr{\url{http://pub.ist.ac.at/~adelcampo/}}
\author[Sottile]{Frank Sottile}
\address{Frank Sottile \\
         Department of Mathematics\\
         Texas A\&M University\\
         College Station\\
         Texas \ 77843\\
         USA}
\email{sottile@math.tamu.edu}
\urladdr{\url{http://www.math.tamu.edu/~sottile}}
\thanks{Research supported in part by NSF grant DMS-915211 and the Institut
  Mittag-Leffler}
\subjclass[2010]{14N15, 05E15}
\keywords{Galois groups, Schubert calculus, Kostka numbers, Enumerative geometry}
\begin{document}

\begin{abstract}
 We show that the Galois group of any Schubert problem involving lines in projective space
 contains the alternating group.  
 This constitutes the largest family of enumerative problems whose Galois groups have been
 largely determined. 
 Using a criterion of Vakil and a special position argument due to Schubert,
 our result follows from a particular inequality among Kostka
 numbers of two-rowed tableaux.
 In most cases, a combinatorial injection proves the inequality.
 For the remaining cases, we use 
 the Weyl integral formulas 
 to obtain an integral formula for these Kostka numbers.
 This rewrites the inequality as an integral, which we estimate to establish the inequality.
\end{abstract}

\maketitle

%
\section*{Introduction}

Galois (monodromy) groups of problems from enumerative geometry were first treated by Jordan in 1870~\cite{J1870}, who studied  several classical 
problems with intrinsic structure, showing that their Galois group was not the full
symmetric group on the set of solutions to the enumerative problem.  
Others~\cite{DBM,Weber} refined this work, which focused on the equations for the
enumerative problem.
%
%
%
Earlier, Hermite gave a different connection to geometry, showing that the algebraic
  Galois group coincided with a geometric monodromy group~\cite{Hermite} in the context of
  Puiseaux fields and algebraic curves.
This line of inquiry remained dormant until a 1977 letter of Serre to
Kleiman~\cite[p.~325]{Kl87}. 
The modern, geometric, theory began with  Harris~\cite{Ha79}, who 
determined the Galois groups of several classical problems, including many whose Galois group
is equal to the full symmetric group.   
In general, we expect that the Galois group of an enumerative problem is the full
symmetric group and when it is not, the geometric problem possesses some intrinsic
structure. 
Despite this, there are relatively few enumerative problems whose Galois group is known.
For a discussion, see Harris~\cite{Ha79} and Kleiman~\cite[pp.~356-7]{Kl87}.

The Schubert calculus of enumerative geometry~\cite{KL72} is a method to compute the
number of solutions to \demph{Schubert problems}, which are a class of geometric
problems involving linear subspaces. 
The algorithms of Schubert calculus reduce the enumeration to combinatorics.
For example, the number of solutions to a Schubert problem involving lines
is a Kostka number for a rectangular partition with two parts.
This well-understood class of problems provides a laboratory with which to study Galois groups of
enumerative problems.

The prototypical Schubert problem is the classical problem of four lines, which asks for
the number of lines in space that meet four given lines. 
To answer this, note that three general lines $\Red{\ell_1},\Blue{\ell_2}$, and $\Green{\ell_3}$
lie on a unique doubly-ruled hyperboloid, shown in Figure~\ref{F:four_lines}. 
\begin{figure}[htb]
\[
  \begin{picture}(314,192)
   \put(3,0){\includegraphics[height=6.6cm]{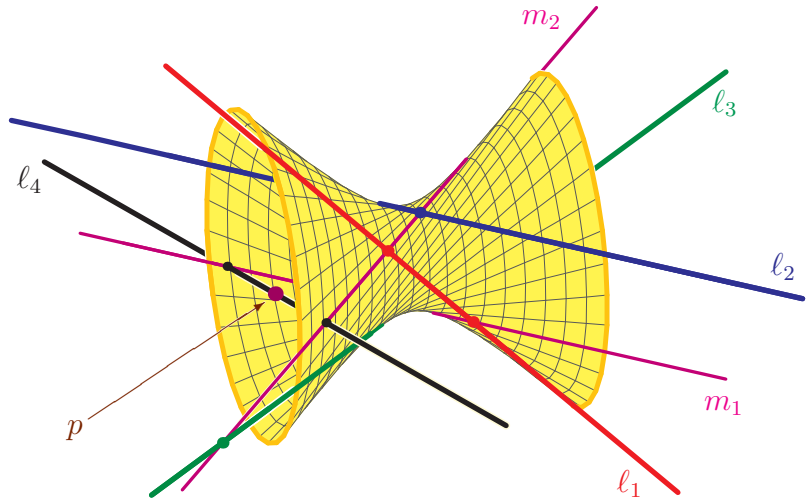}}
   \put(230,  3){\Red{$\ell_1$}}
   \put(288, 85){\Blue{$\ell_2$}}
   \put(266,147){\ForestGreen{$\ell_3$}}
   \put(  3,118){$\ell_4$}
   \put(263, 34){\Magenta{$m_1$}}
   \put(194,180){\Magenta{$m_2$}}

   \thicklines
   \put(30,30.5){\White{\line(3,2){66.6}}}
   \thinlines
   \put( 22, 25){\Brown{$p$}}\put(30,30.5){\Brown{\vector(3,2){66.6}}}
  \end{picture}
\]
\caption{The two lines meeting four lines in space.\label{F:four_lines}}
\end{figure}
These three lines lie in one ruling, while the second ruling consists of
the lines meeting  $\Red{\ell_1},\Blue{\ell_2}$, and $\Green{\ell_3}$.
The fourth line $\ell_4$ meets the hyperboloid in two points.
Each of these points determines a line in the second
ruling, giving two lines \Magenta{$m_1$} and \Magenta{$m_2$} which meet our four given lines.
In terms of Kostka numbers, enumerating the solutions is equivalent to enumerating the 
tableaux of shape $\lambda = (2,2)$ with content $(\Red{1},\Blue{1},\Green{1},1)$. 
There are two such tableaux:
\[
\hspace{15pt}
\begin{picture}(56,27)(-3.4,-3)
  \put(-3.4,-3){\includegraphics{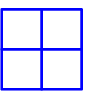}}
  \put(0,11.5){\small  \Red{1}}\put(11.5,11.5){\small \Blue{2}}
  \put(0, 0){\small \Green{3}}\put(11.5, 0){\small 4} 
\end{picture}
\hspace{-10pt}
\begin{picture}(56,27)(-3.4,-3)
  \put(-3.4,-3){\includegraphics{pictures/22.blue.eps}}
  \put(0,11.5){\small  \Red{1}}\put(11.5,11.5){\small \Green{3}}
  \put(0, 0){\small \Blue{2}}\put(11.5, 0){\small 4} 
\end{picture}
\]

When the field is the complex numbers, 
Hermite's result gives one approach to studying the
Galois group---by directly computing monodromy.
For instance, the Galois group of the problem of four lines is the group of permutations
which are obtained by following the solutions over closed paths in the space of lines
$\Red{\ell_1},\Blue{\ell_2},\Green{\ell_3}, \ell_4$.
Rotating $\ell_4$ about the point $p$ (shown in
Figure~\ref{F:four_lines}) gives a closed path which   
interchanges the two solution lines $\Magenta{m_1}$ and $\Magenta{m_2}$, showing that the
Galois group is the full symmetric group on the two solutions.

Leykin and Sottile~\cite{LS09} followed this approach, using numerical homotopy
continuation~\cite{SW05} to compute monodromy for a few dozen so-called {\it simple} Schubert problems,
showing that in each case the Galois group was the full symmetric group on the set of
solutions. 
(The problem of four lines is simple.)
This included a problem involving 2-planes in $\P^8$ with 17589 solutions.
Billey and Vakil~\cite{BV} 
used elimination theory to
compute lower bounds of Galois groups, and they 
showed
that a few enumerative problems on Grassmannians with at most 10 solutions have Galois group
equal to the full symmetric group. 

When the ground field is algebraically closed, 
Vakil~\cite{Va06b} gave a combinatorial criterion,
based on classical special position arguments and group theory, which can be used
recursively to show that a Galois group contains the alternating group on its set of
solutions. 
He used this and his geometric Littlewood-Richardson rule~\cite{Va06a} to show that the
Galois group of every Schubert problem involving lines in  projective space $\P^n$ for
$n<16$ had Galois group that was at least alternating. 
Our main result is based on this observation.

\begin{theorem}\label{Th:one}
 The Galois group of any Schubert problem involving lines in $\P^n$ contains the 
  alternating group on its set of solutions.
\end{theorem}

This nearly determines the Galois group for a large class of Schubert problems. 
In Subsection~\ref{S:symmetric}, we present two infinite families of Schubert problems of
lines, both of which generalize the problem of four lines, and show that each Schubert
problem in these families has Galois group the full symmetric group on its set of solutions.
We conjecture this is always the case for Schubert problems of lines.

\begin{conjecture}
 Any Schubert problem involving lines in $\P^n$ has Galois group the full symmetric
   group on its set of solutions.
\end{conjecture}

This conjecture (and the result of Theorem~\ref{Th:one}) does not hold for Schubert
problems in general. 
Vakil, and independently, Derksen, gave a  Schubert problem in the Grassmannian of
4-planes in 8-dimensional space whose Galois group is not the
full symmetric group on its set of solutions~\cite[\S 3.12]{Va06b}.
In~\cite{RSSS} a Schubert problem with such a deficient Galois group was found in the
manifold of flags in 6-dimensional space.
Both examples generalize to infinite families of Schubert problems with deficient Galois
groups. \smallskip

By Vakil's criterion and a special position argument of Schubert, Theorem~\ref{Th:one} reduces
to a certain inequality among Kostka numbers of two-rowed tableaux.  
For most cases, the inequality follows from a
combinatorial injection of Young tableaux. 
For the remaining cases, 
we use representation theory to rewrite these Kostka numbers as certain
trigonometric integrals~\eqref{Eq:Kostka_integral}.
In this way, the inequalities of Kostka numbers become inequalities of integrals,
which we establish using only elementary calculus.

In Section~\ref{S:one} we give some background on Galois groups, Vakil's criterion, and the
Schubert calculus of lines.
Section~\ref{S:two} we explain Schubert's recursion and formulate our proof of
Theorem~\ref{Th:one}, 
showing that it follows from an inequality of Kostka numbers, which we prove for
most Schubert problems.
We study Kostka numbers in Section~\ref{S:Kostka}, 
giving combinatorial formulas for some and the integral
  formula~\eqref{Eq:Kostka_integral}.
The technical heart of this paper is Section~\ref{S:three} in which we use
these formulas for Kostka numbers
to establish the inequality when $a_1=\dotsb=a_m=a$, which completes the
proof of Theorem~\ref{Th:one}.
  
%
\section{Background}\label{S:one}

%
\subsection{Galois groups and Vakil's criterion}

We summarize Vakil's presentation in~\cite[\S~5.3]{Va06b}.
Suppose that $\pr\colon W\to X$ is a dominant morphism of (generic) degree $d$ between
irreducible algebraic varieties of the same dimension defined over an algebraically closed
field $\K$.
We will assume here and throughout that $\pr$ is generically separable in that the
corresponding extension $\pr^*(\K(X))\subset\K(W)$ of function fields is separable.
Consider the following subscheme of the fiber product
\[
  W^{(d)}\ :=\ 
   (\underbrace{W \times_X \cdots \times_X W}_{d}) \setminus \Delta\,,
\]
where $\Delta$ is the big diagonal. 
Let $x\in X$ be a point where $\pr^{-1}(x)$ consists of $d$ distinct points,
  $\{w_1,\dotsc,w_d\}$.
Then the fiber of $W^{(d)}$ over $x$ consists of all permutations of those points,
\[
    \{ (w_{\sigma(1)},\dotsc,w_{\sigma(d)})\,\mid\, \sigma\in\calS_d\}\,,
\]
where $\calS_d$ is the symmetric group on $d$ letters.
The \demph{Galois group $\Gal_{W\to X}$} is the group of permutations $\sigma\in\calS_d$ for
which $(w_1,\dotsc,w_d)$ and $ (w_{\sigma(1)},\dotsc,w_{\sigma(d)})$ lie in the same
component of $W^{(d)}$.
The Galois group  $\Gal_{W\to X}$ is \demph{deficient} if it is not the full symmetric group
$\calS_d$, and it is \demph{at least alternating} if it is $\calS_d$ or its alternating
subgroup.

Vakil's criterion addresses how $\Gal_{W\to X}$ is affected by 
the Galois group of a restriction of $\pr\colon W\to X$ to a subvariety $Z\subset X$.
Suppose that we have a fiber diagram
 \begin{equation}\label{Eq:fiber_diagram}
  \raisebox{-22.5pt}{\begin{picture}(60,45)
   \put(5,35){$Y$} \put(15,35){$\lhra$} \put(47,35){$W$}
   \put(-5,19){$\pr$}\put(9,32){\vector(0,-1){20}}
      \put(52,32){\vector(0,-1){20}}\put(54,19){$\pr$}
   \put(5, 0){$Z$} \put(15, 0){$\lhra$} \put(47, 0){$X$}
  \end{picture}}
 \end{equation}
where $Z\hookrightarrow X$ is the closed embedding of a Cartier divisor $Z$ of $X$, $X$
is smooth in codimension one along $Z$,
and $\pr\colon Y\to Z$ is a generically separable, dominant morphism of degree $d$.
When $Y$ is either irreducible or has two components, we have the following.
\begin{enumerate}
 \item[(a)] If $Y$ is irreducible, then there is an inclusion $\Gal_{Y\to Z}$ into $\Gal_{W\to X}$.
 \item[(b)] If $Y$ has two components, $Y_1$ and $Y_2$, each of which maps dominantly to $Z$ of
   respective degrees $d_1$ and $d_2$, then there is a subgroup $H$ of 
   $\Gal_{Y_1\to Z}\times \Gal_{Y_2\to Z}$ which maps surjectively onto each factor
   $\Gal_{Y_i\to Z}$ and which 
   includes into $\Gal_{W\to X}$ (via $\calS_{d_1}\times\calS_{d_2}\hookrightarrow\calS_d$). 
\end{enumerate}

Vakil's Criterion follows by  purely group-theoretic arguments including
Goursat's Lemma.

\begin{VC}
 In {\rm Case $(a)$}, if $\Gal_{Y\to Z}$ is at least alternating, then $\Gal_{W\to X}$ is at least
 alternating.
 In {\rm Case $(b)$}, if $\Gal_{Y_1\to Z}$ and $\Gal_{Y_2\to Z}$ are at least alternating, and
 if either $d_1\neq d_2$ or $d_1=d_2=1$, then  $\Gal_{W\to X}$ is at least alternating.
\end{VC}

\begin{remark}\label{R:Vakil}
 This criterion applies to more general inclusions $Z\hookrightarrow X$ of an irreducible
 variety into $X$.
 All that is needed is that $X$ is generically smooth along $Z$, for then we may replace $X$
 by an affine open set meeting $Z$ and there are subvarieties
 $Z=Z_0\subset Z_1\subset\dotsb\subset Z_m=X$ with each inclusion $Z_{i-1}\subset Z_i$ that of a
 Cartier divisor where $Z_i$ is smooth in codimension one along $Z_{i-1}$.
 \hfill\qed
\end{remark}

%
\subsection{Schubert problems of lines}

Let $\G(1,\P^n)$ (or simply $\G(1,n)$) be the Grassmannian of lines in $n$-dimensional
projective space $\P^n$, which is an algebraic manifold of dimension $2n{-}2$.
A \demph{Schubert subvariety} is the set of lines incident on a flag of linear subspaces
$L\subset\Lambda\subset\P^n$, 
 \begin{equation}\label{Eq:SchubertVariety}
   \defcolor{\Omega(L{\subset}\Lambda)}\ :=\ \{\ell\in\G(1,n)\mid \ell\cap L\neq\emptyset
    \quad\mbox{\rm and}\quad \ell\subset\Lambda\}\,.
 \end{equation}
A \demph{Schubert problem} asks for the lines incident on a fixed, but general
collection of flags $L_1{\subset}\Lambda_1,\dotsc,L_m{\subset}\Lambda_m$.
This set of lines is described by the intersection of Schubert varieties
 \begin{equation}\label{Eq:SchubertProblem}
  \Omega(L_1{\subset}\Lambda_1)\,\cap\, \Omega(L_2{\subset}\Lambda_2)
   \,\cap\, \dotsb\,\cap\, \Omega(L_m{\subset}\Lambda_m)\,.
 \end{equation}
Schubert~\cite{Sch1886a} gave a recursion
%
for determining the number of solutions to
a Schubert problem in $\G(1,\P^n)$, when there are finitely many solutions.
The geometry behind his recursion is central to our proof of Theorem~\ref{Th:one}, and we
will present it in Subsection~\ref{S:SchubertRecursion}.

\begin{remark}\label{R:special_only}
 When $\Lambda=\P^n$, we may omit $\Lambda$ and write
 $\defcolor{\Omega_L}:=\Omega(L{\subset}\P^n)$, which is a \demph{special Schubert variety}.
 Note that $\Omega(L{\subset}\Lambda)=\Omega_L$, the latter considered as a subvariety of
 $\G(1,\Lambda)$. 
 Given $L{\subset}\Lambda$ and $L'{\subset}\Lambda'$, if we set $M:=L\cap\Lambda'$ and 
 $M':=L'\cap\Lambda$, then 
 \[
   \Omega(L{\subset}\Lambda)\,\cap\,\Omega(L'{\subset}\Lambda')\ =\ 
   \Omega_M\cap \Omega_{M'}\,,
 \]
 the latter intersection taking place in $\G(1,\Lambda\cap\Lambda')$.

 Given a Schubert problem~\eqref{Eq:SchubertProblem}, if
 $\Lambda:=\Lambda_1\cap\dotsb\cap\Lambda_m$ and 
 $L'_i:=L_i\cap\Lambda$, for $i=1,\dotsc,m$, then we may rewrite~\eqref{Eq:SchubertProblem}
 as an intersection in  $\G(1,\Lambda)$,
 \[
    \Omega_{L'_1}\,\cap\,\Omega_{L'_2}\,\cap\,\dotsb\,\cap\,\Omega_{L'_m}\,.
 \]
 We will show that it suffices
  to study intersections of special Schubert varieties.
 \hfill\qed
\end{remark}
Suppose that $\dim L =n{-}1{-}a$.
A general line in $\Omega_L$ determines and is determined by its intersections with $L$ and with
a fixed hyperplane $H$ not containing $L$.
Thus $\Omega_L$ has dimension
\[
   \dim H+\dim L\ =\ n{-}1+n{-}1{-}a\ =\ 2n{-}2{-}a\ =\ 
   \dim \G(1,n){-}a\,,
\]
and so it has codimension $a$ in $\G(1,n)$.
If $L_1,\dotsc,L_m$ are general linear subspaces of $\P^n$ with $\dim L_i=n{-}1{-}a_i$ for
$i=1,\dotsc,m$, and $a_1+\dotsb+a_m=2n{-}2=\dim\G(1,n)$, then the intersection
 \begin{equation}\label{Eq:special_intersection}
   \Omega_{L_1}\,\cap\,\Omega_{L_2}\,\cap\,\dotsb\,\cap\,\Omega_{L_m}
 \end{equation}
is transverse and therefore zero-dimensional.
Over fields of characteristic zero, transversality follows from Kleiman's Transversality
Theorem~\cite{Kl74} while in positive characteristic, it is Theorem~E in~\cite{So97}.
By this transversality, the number of points in the
intersection~\eqref{Eq:special_intersection} does not depend upon the choice of general
$L_1,\dotsc,L_m$, but only on the numbers $(a_1,\dotsc,a_m)$. 
We call $\defcolor{\adot}:=(a_1,\dotsc,a_m)$ the \demph{type} of the Schubert
intersection~\eqref{Eq:special_intersection}. 

Observe that we do not need to specify $n$.
Given positive integers $\adot=(a_1,\ldots,a_m)$ whose sum is even, 
set $\defcolor{n(\adot)}:=\frac{1}{2}(a_1+\cdots+a_m+2)$.
Henceforth, a Schubert problem will be denoted by a list $\adot$ of positive integers with
even sum. 
It is \demph{valid} if $a_i\leq n(\adot){-}1$ (this is forced by $\dim L_i\geq 0$), which
is equivalent to the numbers $a_1,\ldots,a_m$ being the sides of a (possibly degenerate) polygon.
If $\adot$ is a valid Schubert problem, then we set \defcolor{$K(\adot)$} to be the number
of points in a general intersection~\eqref{Eq:special_intersection} of type $\adot$,
and if $\adot$ is invalid, we set $K(\adot):=0$.

This intersection number $K(\adot)$ is a \demph{Kostka number}, which is
the number of Young tableaux of shape $(n(\adot){-}1,n(\adot){-}1)$ and content
$(a_1,\ldots,a_m)$~\cite[p.25]{Fu97}.   
If $\adot$ is invalid, then there are no such tableaux, which is consistent with our
declaration that $K(\adot)=0$.
These are arrays consisting of two rows of integers, each of length
$n(\adot){-}1$ such that the integers
increase weakly across each row and strictly down each column, and
there are $a_i$
occurrences of $i$ for each $i=1,\ldots,m$.
Let \defcolor{$\calK(\adot)$} be the set of such tableaux.
For example, here are the five Young tableaux in $\calK(2,2,1,2,3)$,
showing that $K(2,2,1,2,3)=5$.
\begin{equation}\label{Eq11:tableaux}
 \raisebox{-9pt}{%
   \Kosfi{1}{1}{2}{2}{3}{4}{4}{5}\qquad
   \Kosfi{1}{1}{2}{2}{4}{3}{4}{5}\qquad
   \Kosfi{1}{1}{2}{3}{4}{2}{3}{5}\qquad
   \Kosfi{1}{1}{2}{4}{4}{2}{3}{5}\qquad
   \Kosfi{1}{1}{3}{4}{4}{2}{2}{5}}
\end{equation}

%
\subsection{Reduced Schubert problems}\label{SS:reduced}
It suffices to consider only certain types of Schubert problems.
Let $\adot$ be a (valid) Schubert problem with $a_1+a_2\geq n(\adot)$ and set
$\defcolor{n}:=n(\adot)$. 
Suppose that $L_1,\dotsc,L_m\subset\P^{n}$ are general linear subspaces with 
$\dim L_i=n{-}1{-}a_i$ for $i=1,\dotsc,m$.
Since $a_1+a_2>n{-}1$, the subspaces $L_1$ and $L_2$ are disjoint, and so every line
$\ell$ in 
\[
  \Omega_{L_1}\cap \Omega_{L_2}\ =\ 
    \{\ell\in\G(1,n)\mid \ell\cap L_i\neq \emptyset\ \mbox{for}\ i=1,2\}
\]
is spanned by its intersections with $L_1$ and $L_2$.
Thus $\ell$ lies in the linear span $\overline{L_1,L_2}$,
which is a proper linear subspace of $\P^n$.
Let $\Lambda$ be a general hyperplane containing $\overline{L_1,L_2}$.

If we set $\defcolor{L'_i}:=L_i\cap\Lambda$ for $i=1,\dotsc,m$, then we have
 \begin{equation}\label{Eq:reduced_Schubert_Problem}
   \Omega_{L_1}\cap \Omega_{L_2}\cap\dotsb\cap \Omega_{L_m}\ =\ 
     \Omega_{L'_1}\cap \Omega_{L'_2}\cap\dotsb\cap \Omega_{L'_m}\,,
 \end{equation}
the latter intersection in $\G(1,\Lambda)\simeq\G(1,n{-}1)$.
For $i=1,2$, we have $L'_i=L_i$ and so
\[
   \dim L'_i\ =\ n{-}1{-}a_i\ =\ (n{-}1){-}1{-}(a_i{-}1)\ =\ 
     \dim \Lambda {-}1{-}(a_i{-}1)\,,
\]
and if $i>2$, then 
 \begin{equation}\label{Eq:reduction}
   \dim L'_i\ =\ n{-}1{-}a_i-1\ =\ (n{-}1){-}a_i\ =\ 
     \dim \Lambda {-}1{-}a_i\,.
 \end{equation}
Thus the righthand side of~\eqref{Eq:reduced_Schubert_Problem} is a Schubert problem of type 
$\defcolor{\adot'}:=(a_1{-}1,a_2{-}1,a_3,\dotsc,a_m)$, and so we have 
\[
   K(a_1,\dotsc,a_m)\ =\ K(a_1{-}1,a_2{-}1,a_3,\dotsc,a_m)\,,
\]
where $a_1'+a_2'-n(\adot')<a_1+a_2-n(\adot)$.
We may also see this combinatorially: the condition $a_1+a_2\geq n(\adot)$ implies that the
first column of every tableaux in $\calK(\adot)$ consists of a 1 on top of a 2.  
Removing this column gives a tableaux in $\calK(\adot')$, and this defines a bijection
between these two sets of tableaux.

We say that a Schubert problem $\adot$ is \demph{reduced} if 
$a_i+a_j< n(\adot)$ for any $i<j$.
Applying the previous procedure recursively shows that 
every Schubert problem may be recast as an equivalent reduced Schubert problem.

%
\subsection{Galois groups of Schubert problems}\label{SS:Galois}

Given a Schubert problem $\adot$, let $n:=n(\adot)$, and set
\[
  \defcolor{X}\ :=\ \{(L_1,\dotsc,L_m)\mid L_i\subset \P^n
  \mbox{\ is a linear space of dimension }n{-}1{-}a_i\}\,,
\]
which is a product of Grassmannians, and hence smooth.
Consider the \demph{total space} of the Schubert problem $\adot$, 
\[
  \defcolor{W}\ := \{(\ell,L_1,\dotsc,L_m)\in  \G(1,n)\times X \mid 
     \ell\cap L_i\neq\emptyset\,,\ i=1,\dotsc,m\}\,.
\]
The projection map $W\to\G(1,n)$ to the first coordinate realizes $W$ as a fiber bundle
of $\G(1,n)$ with irreducible fibers.
As $\G(1,n)$ is irreducible, $W$ is irreducible.

Let $\pr\colon W\to X$ be the other projection.
Its fiber over a point $(L_1,\dotsc,L_m)\in X$ is
 \begin{equation}\label{Eq:fiber}
   \pr^{-1}(L_1,L_2,\dotsc,L_m)\ =\ \Omega_{L_1}\cap \Omega_{L_2}\cap\dotsb\cap \Omega_{L_m}\,.
 \end{equation}
In this way, the map $\pr\colon W\to X$ contains all intersections of Schubert varieties of
type $\adot$.
As the general Schubert problem is a transverse intersection containing $K(\adot)$
points, $\pr$ is generically separable, and it is a dominant (in fact surjective) map of
degree $K(\adot)$.  

\begin{definition}
  The Galois group \defcolor{$\Gal(\adot)$} of the Schubert problem of type $\adot$
  is the Galois group $\Gal_{W\to X}$,
  where $W\to X$ is the projection $\pr$ defined above.
 \hfill\qed
\end{definition}

\begin{remark}\label{R:reductionsOK}
 These two reductions,
 that a general Schubert problem on $\G(1,n)$ is equivalent
 to one that only involves special Schubert varieties (Remark~\ref{R:special_only}), and is
 furthermore equivalent to a reduced Schubert problem (Subsection~\ref{SS:reduced}), do not
 affect the corresponding Galois groups.
 The reason is the same for both reductions, so we only explain it for that of
 Subsection~\ref{SS:reduced}. 

 Suppose that $\adot=(a_1,\dotsc,a_m)$ is a valid, but non-reduced Schubert problem with
 $a_1+a_2\geq n:=n(\adot)$.
 Let $\pr\colon W\to X$ be the family of all instances of Schubert problems of type $\adot$
 as above.
 Fix a hyperplane $\Lambda\subset\P^n$ and let $\defcolor{Z}\subset X$ be
\[  
   \{(L_1,\dotsc,L_m)\in X\,\mid\, L_1,L_2\subset\Lambda\}\,,
\]
 which is smooth.
 By setting $\defcolor{Y}:=W|_Z$ we obtain a fiber diagram as
 in~\eqref{Eq:fiber_diagram} where $Y\to Z$ is the family of all
 Schubert problems of type $\adot'=(a_1{-}1,a_2{-}1,a_3,\dotsc,a_m)$ in $\G(1,\Lambda)$,
 as in Subsection~\ref{SS:reduced}.

 As $Y$ is irreducible we have inclusions $\Gal_{Y\to Z}\hookrightarrow\Gal_{W\to X}$ and
 therefore $\Gal(\adot')\subset\Gal(\adot)$.
 Thus, if $\Gal(\adot')$ is at least alternating, then $\Gal(\adot)$ is at least
 alternating.  

 Moreover, these Galois groups coincide. 
 Note that $\PGL(n{+}1)$ acts on $\P^n$ and thus diagonally on $X$ and 
 the orbit of $Z$ is dense in $X$.
 This action extends to $W\to X$ and to $W^{(d)}\to X$.
 Thus if $z\in Z$ is a point where $\pr^{-1}(z)$ consists of $d$ distinct points, 
 $\{w_1,\dotsc,w_d\}$ and $\sigma$ is a permutation, then the points
 $(w_1,\dotsc,w_d)$ and $ (w_{\sigma(1)},\dotsc,w_{\sigma(d)})$ lie in the same connected
 component of $Y^{(d)}$ if and only if they lie in the same connected component of $W^{(d)}$.  
\hfill\qed
\end{remark}

%
\section{Galois groups of Schubert problems of lines}\label{S:two}

We explain how a special position argument of Schubert together with
Vakil's criterion reduces the proof of Theorem~\ref{Th:one} to establishing an inequality
of Kostka numbers.
In many cases, the inequality follows from simple counting.
The remaining cases are treated in Section~\ref{S:three}.
We also give two infinite families of Schubert problems whose Galois groups
are the full symmetric groups.

%
\subsection{Schubert's degeneration}\label{S:SchubertRecursion}

We begin with a simple observation due to Schubert~\cite{Sch1886a}.

\begin{lemma}\label{L:schubert}
 Let $b_1, b_2$ be positive integers with $b_1+b_2\leq n{-}1$,
 and suppose that $M_1,M_2\subset\P^n$ are linear subspaces with 
 $\dim M_i=n{-}1{-}b_i$ for $i=1,2$.
 If $M_1$ and $M_2$ are in special position in that their linear span is a hyperplane
 $\Lambda=\overline{M_1,M_2}$, then  
 \begin{equation}\label{Eq:degeneration}
   \Omega_{M_1}\cap \Omega_{M_2}\ =\ 
    \Omega_{M_1\cap M_2}\ \bigcup\ \Omega(M_1{\subset}\Lambda)\cap \Omega_{M'_2}\,,
 \end{equation}
 where $M'_2$ is any linear subspace of dimension $n{-}b_2$ of\/ $\P^n$ with
 $M'_2\cap\Lambda=M_2$.
 Furthermore, the intersection $\Omega_{M_1}\cap \Omega_{M_2}$ is generically
 transverse,~\eqref{Eq:degeneration} is its irreducible decomposition, and  
 the second intersection of Schubert varieties is also generically transverse.
\end{lemma}

The reason for this decomposition is that if $\ell$ meets both $M_1$ and $M_2$, then either it
meets $M_1\cap M_2$ or it lies in their linear span (while also meeting both $M_1$ and $M_2$).
This lemma, particularly the transversality statement, is proven in~\cite[Lemma~2.4]{So97}.

\begin{remark}\label{R:Schubert_recursion}
 Suppose that $\adot$ is a reduced Schubert problem.
 Set $n:=n(\adot)$.
 Let $L_1,\dotsc,L_m$ be linear subspaces with $\dim L_i=n{-}a_i{-}1$ which are in general
 position in $\P^n$, except that $L_{m-1}$ and $L_m$ span a hyperplane $\Lambda$.
 By Lemma~\ref{L:schubert} we have 
 \begin{multline}\label{Eq:Y1Y2}
  \quad
   \Omega_{L_1}\cap\dotsb\cap \Omega_{L_m}\ =\,\  
   \Omega_{L_1}\cap\dotsb\cap \Omega_{L_{m-2}}\;\cap\; \Omega_{L_{m-1}\cap L_m}\ \\
   \bigcup\ 
   \Omega_{L_1}\cap\dotsb\cap \Omega_{L_{m-2}}\;\cap\; 
     \Omega(L_{m-1}{\subset}\Lambda)\cap \Omega_{L'_m}\,,
  \quad
 \end{multline}
 where $L'_m\cap\Lambda=L_m$, and so $L'_m$ has dimension $n{-}a_m$.

 The first intersection on the righthand side of~\eqref{Eq:Y1Y2} has type
 $(a_1,\dotsc,a_{m-2},a_{m-1}{+}a_m)$ and the second, once we apply the reduction of
 Remark~\ref{R:special_only}, has type $(a_1,\dotsc,a_{m-2},\, a_{m-1}{-}1,a_m{-}1)$. 
 This gives Schubert's recursion for Kostka numbers
 \begin{equation}\label{E:combinatorial_recursion}
  K(a_1,\dotsc,a_m)\ =\ 
    K(a_1,\dotsc,a_{m-2},\,a_{m-1}+a_m)\ +\ 
    K(a_1,\dotsc,a_{m-2},\, a_{m-1}{-}1,a_m{-}1)\,.
 \end{equation}
\end{remark}

As $\adot$ is reduced, the two Schubert problems obtained are both valid.
This recursion holds even if $\adot$ is not reduced.
The first term in~\eqref{E:combinatorial_recursion} may be zero, for
$(a_1,\dotsc,a_{m-2},\,a_{m-1}+a_m)$ may not be valid (in this case, 
$L_{m-1}\cap L_m=\emptyset$).

We consider this recursion for $K(2,2,1,2,3)$.
The first tableau in~(\ref{Eq11:tableaux}) has both 4s in its second row (along with its
5s), while the remaining four tableaux have last column consisting of a 4 on top of a 5.
If we replace the 5s by 4s in the first tableau and erase the last column in the remaining
four tableaux, we obtain
\[
   \Kosfi{1}{1}{2}{2}{3}{4}{4}{4}\qquad\qquad
   \Kosfo{1}{1}{2}{2}{3}{4}{5}\qquad
   \Kosfo{1}{1}{2}{3}{2}{3}{5}\qquad
   \Kosfo{1}{1}{2}{4}{2}{3}{5}\qquad
   \Kosfo{1}{1}{3}{4}{2}{2}{5}\ \,
\]
which shows that $K(2,2,1,2,3)=K(2,2,1,5)+K(2,2,1,1,2)$.
We sometimes use exponential notation for the sequences $\adot$, e.g.\ 
$(1^2,2^3,3)=(1,1,2,2,2,3)$.

In Subsection~\ref{S:counting}, we use this recursion to prove the following lemmas.

\begin{lemma}\label{L:small_m}
 Suppose that  $\adot$ is a valid Schubert problem.
 Then $K(\adot)\neq 0$ and $m>1$.
 If $m=2$ or $m=3$, then $K(\adot)=1$.
 If $m=4$, then
 \begin{equation}\label{Eq:m=4}
   K(\adot)\ =\ 1\ +\ \min\{ a_i\,,\ n(\adot){-}1{-}a_j\mid i,j=1,\dotsc,4\}\,.
 \end{equation}
 There are no reduced Schubert problems with $m< 4$.
 If $\adot$ is reduced and $m=4$, then $a_1=a_2=a_3=a_4$, and we have 
 $K(a^4)=1+a$.
\end{lemma}

\begin{lemma}\label{L:m=5}
 Let $a=2b$ with $b\geq 1$ be even.
 Then
\[ 
   K(a^3,2a)\ =\ 1{+}b
    \qquad\mbox{and}\qquad
   K(a^3,(a{-}1)^2)\ =\ \frac{(5b^2+3b)}{2}\,.
\]
\end{lemma}

\subsection{Proof of Theorem~\ref{Th:one}}
We will use Vakil's criterion and Schubert's degeneration 
to deduce Theorem~\ref{Th:one} from a key combinatorial lemma.
A \demph{rearrangement} of a Schubert problem $(a_1,\dotsc,a_m)$ is simply a listing of the
integers $(a_1,\dotsc,a_m)$ in some order.

\begin{lemma}\label{L:induction}
 Let $\adot$ be a reduced Schubert problem involving $m\geq 4$ integers.
 When $\adot \neq(1,1,1,1)$, it has a rearrangement $(a_1,\dotsc,a_m)$
 such that 
 \begin{equation}\label{Eq:ineq}
     K(a_1,\dotsc,a_{m-2},\,a_{m-1}{+}a_m)\ \neq\ 
    K(a_1,\dotsc,a_{m-2},\, a_{m-1}{-}1,a_m{-}1)\,,
 \end{equation}
 and both terms are nonzero.
 When $\adot=(1,1,1,1)$, this$~\eqref{Eq:ineq}$ is an equality with both terms equal to $1$. 
\end{lemma}

The proof of Lemma~\ref{L:induction} will occupy part of this section
and Section~\ref{S:three}. 
We use it to deduce Theorem~\ref{Th:one}, which we restate in a more precise form.\medskip

\noindent{\bf Theorem~\ref{Th:one}.}
 {\it Let $\adot$ be a Schubert problem on $\G(1,\P^n)$.
      Then $\Gal(\adot)$ is at least alternating.} \medskip

\begin{proof}
 We use a double induction on the dimension $n$ of the ambient projective space and the number
 $m$ of conditions.
 The initial cases are when one of $n$ or $m$ is less than four, for by Lemma~\ref{L:small_m},
 $K(a_1,\dotsc,a_m)\leq 2$ and the trivial subgroups of these small symmetric
 groups are alternating.
 Only in case $\adot=(1,1,1,1)$ with $n=3$ is $K(\adot)=2$.

 Given a non-reduced Schubert problem, the associated reduced Schubert
 problem is in a smaller-dimensional projective space, and so its Galois group is at least
 alternating, by hypothesis.
 We may therefore assume that $\adot$ is a reduced Schubert problem, so that for 
 $1\leq i<j\leq m$, we have $a_i+a_j\leq n{-}1$, where 
 $n:=n(\adot)$.
 Let $\pr\colon W\to X$ be as in Subsection~\ref{SS:Galois}, so that fibers of $\pr$ are
 intersections of Schubert problems~\eqref{Eq:fiber}.
 Recall that $X$ is smooth.
 Define $Z\subset X$ by
 \[
    \defcolor{Z}\ :=\ 
    \{(L_1,\dotsc,L_m)\in X\mid 
      L_{m-1},L_m\mbox{ do not span }\P^n\}\,.
 \]
 This subvariety is proper, for if $L_{m-1},L_m$ are general and $a_{m-1}+a_m\leq n{-}1$,
 they span $\P^n$.

 Let \defcolor{$Y$} be the pullback of the map $\pr\colon W\to X$ along the inclusion
 $Z\hookrightarrow X$.
 By Remark~\ref{R:Schubert_recursion}, $Y$ has two components $Y_1$ and $Y_2$ corresponding to the 
 two components of~\eqref{Eq:Y1Y2}.
 The first component $Y_1$ is the total space of the Schubert problem
 $(a_1,\dotsc,a_{m-2},a_{m-1}{+}a_m)$, and so by induction $\Gal_{Y_1\to Z}$ is at least
 alternating. 
 For the second component $Y_2\to Z$,
 first replace $Z$ by its dense open subset in which $L_{m-1},L_m$ span a hyperplane
 $\overline{L_{m-1},L_m}$.
 Observe that under the map from $Z$ to the space of hyperplanes in $\P^n$ given by
 \[
   (L_1,L_2,\dotsc,L_m)\ \longmapsto\ 
   \overline{L_{m-1},L_m}\,,
 \]
 the fiber of $Y_2\to Z$ over a
 fixed hyperplane $\Lambda$ is the total space of the Schubert problem
 $(a_1,\dotsc,a_{m-2},a_{m-1}{-}1,a_m{-}1)$ in $\G(1,\Lambda)$.  
 Again, our inductive hypothesis and Case (a) of Vakil's criterion (as elucidated in
 Remark~\ref{R:Vakil}) implies that $\Gal_{Y_2\to Z}$ is at least alternating.

 We conclude by an application of Vakil's criterion that $\Gal_{W\to X}$ is at least alternating,
 which proves Theorem~\ref{Th:one}.
\end{proof}

%
\subsection{Some Schubert problems with symmetric Galois group}\label{S:symmetric}
While Theorem~\ref{Th:one} asserts that all Schubert problems involving lines have at least
alternating Galois group, we conjectured that Galois groups of Schubert problems of lines
are always the full symmetric group. 
We present some evidence for this conjecture.

The first non-trivial computation of a Galois group of a Schubert problem that we know of was
for the problem $\adot=(1^6)$ in $\G(1,\P^4)$ where $K(\adot)=5$.
Byrnes and Stevens showed that $\Gal(\adot)$ is the full symmetric
group~\cite{BS_homotopy} and~\cite[\S 5.3]{By89}.
In~\cite{LS09} problems $\adot=(1^{2n-2})$ for $n=5,\dotsc,9$ were shown to have Galois group
the full symmetric group.
Both demonstrations used numerical methods.

We describe two infinite families of Schubert problems, each of which has the full symmetric
group as Galois group.
Both are generalizations of the problem of four lines.
In~\cite[\S 8]{So97}, the Schubert problem $\adot=(1^n,n{-}2)$ in $\G(1,\P^n)$ was studied
to find solutions in finite fields. 
It involves lines meeting a fixed line $\ell$ and $n$ codimension-two planes in $\P^n$.
Fixing the line $\ell$ and all but one codimension-two plane, the lines meeting them form a
rational normal scroll $S_{1,n{-}2}$, parametrized by the intersections of these lines with $\ell$.
A general codimension-two plane will meet the scroll in $n{-}1$ points, each of which gives a
solution to the Schubert problem.
These points correspond to $n{-}1$ points of $\ell$, and thus to a homogeneous degree $n{-}1$
form on $\ell$.
The main consequence of~\cite[\S 8]{So97} is that every such form can arise, which shows
this Schubert problem has Galois group the full symmetric group.

The other infinite family is $((a{-}1)^4)$, which is described in~\cite[\S 8]{So97}.
We use a slightly different description of it in the Grassmannian of two-dimensional linear
subspaces of $2a$-dimensional space, $V$ (which is identical to $\G(1,\P^{2a-1})$).
It involves the $2$-planes meeting four general $a$-planes in $V$.
If the $a$-planes are $H_1,\dotsc,H_4$, then any two are in direct sum.
It follows that $H_3$ and $H_4$ are the graphs of linear isomorphisms
$\varphi_3,\varphi_4\colon H_1\to H_2$.
If we set $\psi:=\varphi_4^{-1}\circ\varphi_3$, then $\psi\in GL(H_1)$.
The condition that these four planes are generic is that $\psi$ has distinct eigenvalues
and therefore exactly $a$ eigenvectors $v_1,\dotsc,v_a\in H_1$, up to scalar multiples.
Then the solutions to the Schubert problem are
\[
   \overline{v_i,\varphi_3(v_i)}\qquad\mbox{for }i=1,\dotsc,a\,.
\]
Every element $\psi\in GL(H_1)$ with distinct eigenvalues may occur, which implies that the 
Galois group is the full symmetric group.

We remark that one may also apply Vakil's Remark 3.8~\cite{Va06b} to these problems to deduce that
their Galois group is the full symmetric group.

%
\subsection{Inequality of Lemma~\ref{L:induction} in most cases}

We give a combinatorial injection on sets of Young tableaux to establish
Lemma~\ref{L:induction} when we have $a_i\neq a_j$ for some $i,j$.

\begin{lemma}\label{L:unequal_inequality}
 Suppose that $\adot=(b_1,\ldots,b_\mu,\alpha,\beta,\gamma)$ is a reduced Schubert problem
 where $\alpha\leq\beta\leq\gamma$ with $\alpha<\gamma$.
 Then
 \begin{equation}\label{Eq:unequal_inequality}
   K(b_1,\ldots,b_\mu,\,\alpha, \beta+\gamma)\ <\ 
   K(b_1,\ldots,b_\mu,\,\gamma, \beta+\alpha)\,.
 \end{equation}
\end{lemma}

To see that this implies Lemma~\ref{L:induction} in the case when $a_i\neq a_j$, for
some $i,j$, we apply Schubert's recursion to obtain two different expressions for $K(\adot)$, 
%
 \begin{multline*}
   \qquad   K(b_1,\ldots,b_\mu,\,\alpha, \beta{+}\gamma)\ + \ 
   K(b_1,\ldots,b_\mu,\,\alpha, \beta{-}1,\gamma{-}1)\\
  \ =\    K(b_1,\ldots,b_\mu,\,\gamma, \beta+\alpha)\ +\ 
   K(b_1,\ldots,b_\mu,\,\gamma, \beta{-}1,\alpha{-}1)\,.\qquad
 \end{multline*}
%
By the inequality~(\ref{Eq:unequal_inequality}), at least one of these expressions involves
unequal terms.
Since all four terms are from valid Schubert problems, none are zero, and so this implies
Lemma~\ref{L:induction} when not all $a_i$ are identical.\qed

\begin{proof}[Proof of Lemma~$\ref{L:unequal_inequality}$]
 We establish the inequality~(\ref{Eq:unequal_inequality}) via a combinatorial injection
 \begin{equation}\label{Eq:injection}
   \iota\ \colon\  \calK(b_1,\ldots,b_\mu,\,\alpha, \beta+\gamma)\ 
   \lhook\joinrel\relbar\joinrel\rightarrow\ 
   \calK(b_1,\ldots,b_\mu,\,\gamma, \beta+\alpha)\,,
 \end{equation}
 which is not surjective.

 Let $T$ be a tableau in $\calK(b_1,\ldots,b_\mu,\,\alpha, \beta+\gamma)$ and let $A$ be its
 sub-tableau consisting of the entries $1,\ldots,\mu$.
 Then the skew tableau $T\setminus A$ has a bloc of $(\mu{+}1)$s of length $a$ at the end of
 its first row and its second row consists of a bloc of $(\mu{+}1)$s of length $\alpha{-}a$
 followed by a bloc of $(\mu{+}2)$s of length $\beta{+}\gamma$. 
 Form the tableau $\iota(T)$ by changing the last row of $T\setminus A$ to 
 a bloc of $(\mu{+}1)$s of length $\gamma{-}a$ followed by a bloc of $(\mu{+}2)$s of length
 $\beta{+}\alpha$. 
 Since $a\leq\alpha<\gamma$, this map is well-defined, and gives the
 inclusion~\eqref{Eq:injection}. 
 We illustrate this schematically.
\[
  T\ =\ 
  \raisebox{-13pt}{\begin{picture}(131,33)
   \put(0,0){\includegraphics{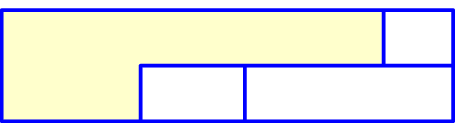}}
   \put(117,22){$a$}
   \put(45,6){$\alpha{-}a$}\put(90,6){$\beta{+}\gamma$}
   \put(17,15){$A$}
  \end{picture}}
   \ \longmapsto\ 
  \raisebox{-13pt}{\begin{picture}(131,33)
   \put(0,0){\includegraphics{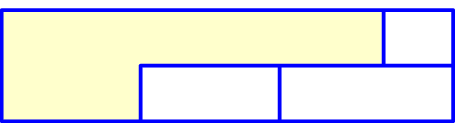}}
   \put(117,22){$a$}
   \put(50,6){$\gamma{-}a$}\put(95,6){$\beta{+}\alpha$}
   \put(17,15){$A$}
  \end{picture}}
  \ =:\ \iota(T)
\]
 
 To show that $\iota$ is not surjective, set
 $\defcolor{b_\bullet}:=(b_1,\ldots,b_\mu,\gamma-\alpha-1,\beta-1)$, which is a valid
 Schubert problem.
 Hence $K(b_\bullet)\neq 0$ and $\calK(b_\bullet)\neq\varnothing$.
 For any $T\in\calK(b_\bullet)$, we may add $\alpha{+}1$ columns to its end consisting of a
 $\mu{+}1$ above a $\mu{+}2$ to obtain a tableau
 $T'\in\calK(b_1,\ldots,b_\mu,\,\gamma,\beta+\alpha)$.
 As $T'$ has more than $\alpha$ $(\mu{+}1)$s in its first row, it is not in the image of the
 injection $\iota$,
 which completes the proof of the lemma.
\end{proof}

%
\section{Some formulas for Kostka numbers}\label{S:Kostka}
 We prove Lemmas~\ref{L:small_m} and~\ref{L:m=5} using Schubert's recursion and
 give an integral formula for Kostka numbers 
 coming from the Weyl integral formula.

%
\subsection{Proof of Lemma~\ref{L:small_m}}\label{S:counting}

We show that if $\adot$ is a valid Schubert problem, then
$K(\adot)\neq 0$, and we also compute $K(\adot)$ for $m\leq 4$.

Observe that there are no valid Schubert problems with $m=1$
(as we require that each component $a_i$ is positive).

%
\subsubsection{}
When $m=2$, valid Schubert problems have the form $(a,a)$ with $n(\adot)=a{+}1$.
The corresponding geometric problem asks for the lines meeting two general linear spaces of
dimension $n{-}a{-}1=0$, that is, the lines meeting two general points.
Thus $K(a,a)=1$.

%
\subsubsection{}
Let $(a,b,c)$ be a valid Schubert problem.
We may assume that $b{+}c>a$ so that $K(a,b,c)=K(a,b{-}1,c{-}1)$ by~\eqref{Eq:reduction}.
Iterating this will lead to a Schubert problem with $m=2$, and so we see that $K(a,b,c)=1$.

%
\subsubsection{}
Suppose that $(a_1,a_2,a_3,a_4)$ is a valid Schubert problem,
and suppose that $a_1\leq a_2\leq a_3\leq a_4$.
If it is reduced, then we have 
\[
  a_3\ +\ a_4\ \leq\ \frac{1}{2}(a_1+a_2+a_3+a_4)\ \leq\ a_3\ +\ a_4\,,
\]
implying that the four numbers are equal, say to $a$.
Write $\adot=(a^4)$ in this case.
By~\eqref{E:combinatorial_recursion}, 
\[
  K(a^4)\ =\ K(a,a,2a)\ +\ K(a,a,a{-}1,a{-}1)\ =\ 
   1\ +\ K((a{-}1)^4)\,,
\]
as $K(a,a,2a)=1$ and $K(a,a,a{-}1,a{-}1)=K((a{-}1)^4)$, by~\eqref{Eq:reduction}.
Since $K(1^4)=2$, as this is the problem of four lines, we obtain
$K(a^4)=1{+}a$,
which proves~\eqref{Eq:m=4} by induction on $a$ when $\adot$ is reduced and therefore equal to
$(a^4)$. 

Now suppose that $\adot$ is not reduced, and set 
 \begin{eqnarray*}
   \defcolor{\alpha(\adot)}&:=&\min\{a_i\mid i=1,\dotsc,4\}\qquad\mbox{and}\\
   \defcolor{\beta(\adot)} &:=&\min\{n(\adot){-}1{-}a_i\mid i=1,\dotsc,4\}\,.
 \end{eqnarray*}
Since $\adot$ is not reduced and $a_1\leq a_2\leq a_3\leq a_4$, we have 
$a_1+a_2< a_1+\dotsb+a_4< a_3+a_4$ and~\eqref{Eq:reduction} gives 
\[
   K(\adot)\ =\ K(a_1,a_2,a_3{-}1,a_4{-}1)\,.
\]
Set $\defcolor{\adot'}:=(a_1,a_2,a_3{-}1,a_4{-}1)$.
We prove~\eqref{Eq:m=4} by showing that 
 \begin{equation}\label{Eq:mins}
   \min\{\alpha(\adot),\beta(\adot)\}\ =\ 
   \min\{\alpha(\adot'),\beta(\adot')\}\,.
 \end{equation}

Note that $n(\adot')=n(\adot){-}1$. 
Since $a_1\leq a_3$, we have $\alpha(\adot')=\alpha(\adot)=a_1$ unless 
$a_1=a_3$, in which case $\adot=(a,a,a,a{+}2\gamma)$ for some $\gamma\geq 1$.
Thus $\adot'=(a{-}1,a,a,a{+}2\gamma{-}1)$, and so $\alpha(\adot')=\alpha(\adot){-}1$.
But then $\beta(\adot')=\beta(\adot)=a{-}\gamma\leq \alpha(\adot')$, which proves~\eqref{Eq:mins}
when $\alpha(\adot')\neq\alpha(\adot)$.

Since $a_2\leq a_4$, we have $\beta(\adot')=\beta(\adot)=n(\adot){-}1{-}a_4$, unless
$a_2=a_4$, in which case $\adot=(a,a{+}2\gamma,a{+}2\gamma,a{+}2\gamma)$ for some 
$\gamma\geq 1$.
Thus $\adot'=(a,a{+}2\gamma{-}1,a{+}2\gamma{-}1,a{+}2\gamma)$, and so 
$\beta(\adot')=\beta(\adot){-}1=a{+}\gamma{-}1$.
But then $\alpha(\adot')=\alpha(\adot)=a\leq\beta(\adot')<\beta(\adot)$, which
proves~\eqref{Eq:mins} 
when $\beta(\adot')\neq\beta(\adot)$, and completes the proof of Lemma~\ref{L:small_m}.

%
\subsection{Proof of Lemma~\ref{L:m=5}}

Let $a=2b$ be positive and even.
By Schubert's recursion~\eqref{E:combinatorial_recursion},
 \[
    K(a^3,(a{-}1)^2) \ =\ K(a^3,2a{-}2)\ +\ K(a^3,(a{-}2)^2)\,.
 \]
If we apply Schubert's recursion to the last term and then repeat, we obtain
 \[
     K(a^3,(a{-}1)^2) \ =\ \sum_{j=1}^a K(a^3, 2a{-}2j)\,.
 \]
 Since $a=2b$ and $n(a^3, 2a{-}2j)=5b-j+1$, Lemma~\ref{L:small_m} implies that 
 \[
     K(a^3, 2a{-}2j)\ =\ 1\ +\ \min\{ 2b, \, 2(2b{-}j),\, 3b{-}j, \, b{+}j\}\,.
 \]
 If $1\leq j\leq b$, then this minimum is $b{+}j$, and if $b< j\leq a=2b$, then this
 minimum is $4b-2j$.
 Writing $j=b+i$ when $b<j$, we have
 \begin{eqnarray*}
   K(a^3,(a{-}1)^2)&=& \sum_{j=1}^b 1{+}b{+}j\ +\ \sum_{i=1}^b 1{+}2b{-}2i\\
   &=& b+b^2+\tfrac{b(b+1)}{2}\ +\ b+2b^2-(b(b+1) \ =\ \frac{5b^2+3b}{2}\,,
 \end{eqnarray*}
which completes the proof of Lemma~\ref{L:m=5}.\hfill\qed

%
\subsection{An integral formula for Kostka numbers}

 Let \defcolor{$V_a$} be the irreducible representation of $SU(2)$ with highest weight $a$.
 Then $K(a_1,\dotsc,a_m)$ is the multiplicity of the trivial representation $V_0$ in the
 tensor product $V_{a_1}\otimes\dotsb\otimes V_{a_m}$.
 If $\chi_a$ is the character of $V_a$, then 
\[
   K(a_1,\dotsc,a_m)\ =\ \langle\chi_0, \prod_i^m\chi_{a_i}\rangle\ =\ 
   \int_{SU(2)} \prod_{i=1}^m \chi_{a_i}(g) dg\,,
\]
 the integral with respect to Haar measure on $SU(2)$, as $\chi_0(g)=1$.

 The Weyl integral formula rewrites this as an integral over the torus $T=U(1)$ of $SU(2)$.
 First note that for $e^{\sqrt{-1}\theta}\in U(1)$,
\[
\chi_{a}(e^{\sqrt{-1}\theta})\ =\ \frac{e^{(a+1)\sqrt{-1}\theta}-e^{-(a+1)\sqrt{-1}\theta}}
{e^{\sqrt{-1}\theta}-e^{\sqrt{-1}\theta}}\ =\ \frac{\sin{(a{+}1)\theta}}{\sin{\theta}} \,  .
\]
 Then the Weyl integral formula gives
 \begin{eqnarray}
  K(a_1,\dotsc,a_m)& =& \nonumber
  2 \int_{0}^{2\pi} 
       \Bigl(\prod_{i=1}^m \frac{\sin{(a_i{+}1)\theta}}{\sin{\theta}}\Bigr)\,
        \sin^2\theta\,\frac{d\theta}{2\pi}\\
   &=& \label{Eq:Kostka_integral}
    \frac{2}{\pi} \int_{0}^{\pi} 
       \Bigl(\prod_{i=1}^m \frac{\sin{(a_i{+}1)\theta}}{\sin{\theta}}\Bigr)\,
        \sin^2\theta\, d\theta\,,
 \end{eqnarray}
as the integrand $f(\theta)$ satisfies $f(\theta)=f(2\pi-\theta)$.

%
\section{Proof of Lemma~\ref{L:induction} when $\adot=(a^m)$}\label{S:three}

We prove Lemma~\ref{L:induction} in the remaining case when $a_1=\dotsb=a_m=a$. 
We use~\eqref{Eq:Kostka_integral} to recast the the inequality of Lemma~\ref{L:induction} into the
non-vanishing of an integral, which we establish by induction. 
It will be convenient to write $\defcolor{\lambda_a(\theta)}$ for the quotient
$\frac{\sin(a+1)\theta}{\sin\theta}$.

%
\subsection{Inequality of Lemma~\ref{L:induction} when $\adot=(a^m)$}

We complete the proof of Theorem~\ref{Th:one} by establishing the inequality of
Lemma~\ref{L:induction} for Schubert problems not covered by Lemma~\ref{L:unequal_inequality}.
For these, every condition is the same, so $\adot=(a,\dotsc,a)=(a^m)$.

If $a=1$, then we may use the hook-length formula~\cite[\S4.3]{Fu97}.
If $\mu+b=2c$ is even, then the Kostka number $K(1^\mu,b)$ is the number of Young tableaux
of shape $(c,c{-}b)$, which is
\[
   K(1^\mu,b)\ =\ \frac{\mu!(b{+}1)}{(c{-}b)!(c{+}1)!}\;.
\]
When $m=2c$ is even, the inequality of Lemma~\ref{L:induction} is
that $K(1^{2c-2})\neq K(1^{2c-2},2)$.
We compute
\[
  K(1^{2c-2})\ =\ \frac{(2c-2)!(1)}{c!(c+1)!}
  \qquad\mbox{and}\qquad
  K(1^{2c-2},2)\ =\ \frac{(2c-2)!(3)}{(c-2)!(c+1)!}
\]
and so 
 \begin{equation}\label{E:a=1}
   K(1^{2c-2},2)/K(1^{2c-2})\ =\ 3\frac{c!(c{+}1)!}{(c{-}2)!(c{+}1)!}
    \ =\ 3\frac{c{-}1}{c{+}1}\ \neq\ 1\,,
 \end{equation}
when $c>2$, but when $c=2$ both Kostka numbers are $1$, which proves the  inequality of
Lemma~\ref{L:induction}, when each $a_i=1$.\smallskip

We now suppose that $\adot=(a^{\mu+2})$ where $a>1$ and $a\mu$ is even.
(We write $m=\mu+2$ to reduce notational clutter.)
The case $a=2$ is different because in the inequality~\eqref{Eq:ineq},
 \[
  K(2^\mu, 4)\ -\ K(2^\mu,1,1)\ \neq\ 0\,,
 \]
the left-hand side is negative for $\mu\leq 13$ and otherwise positive.
This is shown in Table \ref{T:inequality_a=2}.
\begin{table}[htdp]
\caption{The inequality \eqref{Eq:ineq} for the case $\adot=(2^{\mu+2})$.}
\label{T:inequality_a=2}
\begin{tabular}{|c||c|c|c|}
\hline
$\mu$ & $K(2^\mu,4)$ & $K(2^\mu,1,1)$ & Difference \\\hline\hline
2 & 1 & 2 & $-1$ \\\hline
3 & 2 & 4 & $-2$ \\\hline
4 & 6 & 9 & $-3$ \\\hline
5 & 15 & 21 & $-6$ \\\hline
6 & 40 & 51 & $-11$ \\\hline
\vdots & \vdots & \vdots &\vdots \\\hline
%
%
13 & 41262 & 41835 & $-573$ \\\hline
14 & 113841 & 113634 & 207\\\hline
15 & 315420 & 310572 & 4848\\\hline
\end{tabular} 
\end{table}

\begin{lemma}\label{L:inequality_a=2}
  For all $\mu\geq 2$, we have $K(2^\mu,4)\neq K(2^\mu,1,1)$, and both terms are nonzero.
  If $\mu<14$ then $K(2^\mu,4)<K(2^\mu,1,1)$ and if $\mu\geq 14$, then $K(2^\mu,4) >K(2^\mu,1,1)$.
\end{lemma}

The remaining cases $a\geq 3$ have a uniform behavior.

\begin{lemma}\label{L:equal_inequality}
 For $a\geq3$ and for all $\mu\geq 2$ with $a \mu$ even we have
 \begin{equation}\label{E:equal_inequality}
   K(a^\mu,\,2a)\ <\ K(a^\mu,(a{-}1)^2)\,. 
 \end{equation}
\end{lemma}

We establish Lemma~\ref{L:inequality_a=2} in Subsection~\ref{SS:a=2}
and Lemma~\ref{L:equal_inequality} in Subsection~\ref{SS:a=3}.

\begin{proof}[Proof of Lemma~$\ref{L:induction}$ when $\adot=(a^m)$]
 We established the case when $a=1$ by direct computation in~\eqref{E:a=1}.
 Lemma~\ref{L:inequality_a=2} covers the case when $a=2$ as $\mu=m{-}2$,
 and Lemma~\ref{L:equal_inequality} covers the remaining cases.
 This completes the proof of Lemma~\ref{L:induction} and of Theorem~\ref{Th:one}.
\end{proof}

%
\subsection{Proof of Lemma~\ref{L:inequality_a=2}}\label{SS:a=2}

By the computations recorded in Table~\ref{T:inequality_a=2}, we only need to show that 
$K(2^\mu,4)-K(2^\mu,1,1)>0$ for $\mu\geq 14$.
Using~\eqref{Eq:Kostka_integral}, we have
 \begin{eqnarray*}
 \ K(2^\mu,4)-K(2^\mu,1,1)&=& \frac{2}{\pi}\int_0^\pi \lambda_2(\theta)^\mu
          \bigl( \lambda_4(\theta)\ -\ \lambda_1(\theta)^2\bigr)\; \sin^2{\theta})\;
          d\theta\\
  &=&\frac{2}{\pi}\int_0^\pi  \lambda_2(\theta)^\mu 
          \bigl( \sin{5\theta}\ \sin{\theta}\ -\ \sin^2{2\theta}\bigr)\; d\theta\ .\qquad
 \end{eqnarray*}
The integrand $f(\theta)$ of the last integral is symmetric about $\theta=\frac{\pi}{2}$ 
in that $f(\theta)=f(\pi-\theta)$.
Thus it  suffices to prove that if $\mu\geq 14$, then
 \begin{equation}\label{Eq:integral_a=2}
  \int_0^{\pi/2} \lambda_2(\theta)^\mu (\sin{5\theta}\ \sin{\theta}\ -\ \sin^2{2\theta})
   \, d\theta
   \ >\ 0\,.
\end{equation}
To simplify our notation, set
 \[
   \defcolor{F(\theta)}\ :=\  \sin{5\theta}\ \sin{\theta}\ -\ \sin^2{2\theta}\,.
 \]
We graph these functions and the integrand in~\eqref{Eq:integral_a=2} for $\mu=8$
in Figure~\ref{F:graphs_a=2}.
\begin{figure}[htb]

   \begin{picture}(137,163)(-7,0)
    \put(0,0){\includegraphics{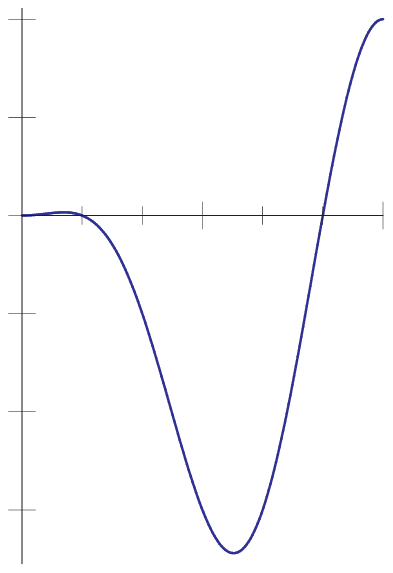}}
    \put(114,85){$\frac{\pi}{2}$}
    \put( 62,85){$\frac{\pi}{4}$}

    \put(-7, 14){$-\frac{3}{2}$}
    \put(-7, 42){$-1$}
    \put(-7, 70){$-\frac{1}{2}$}
    \put( 3,127){$\frac{1}{2}$}
    \put( 3,155){$1$}
    \put(28,10){$F$}
  \end{picture}
 \quad
   \begin{picture}(135,163)(-5,0)
    \put(0,0){\includegraphics{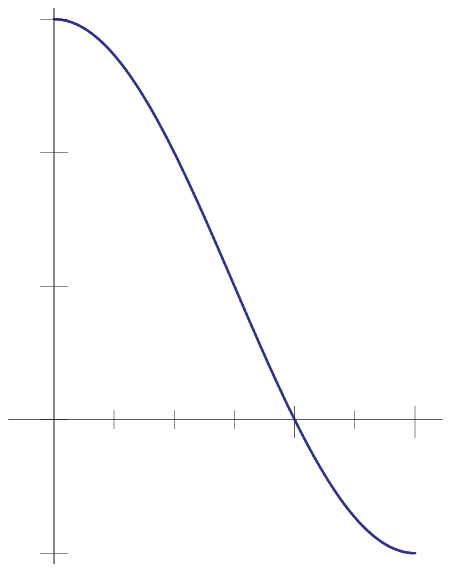}}
    \put(114,27.5){$\frac{\pi}{2}$}
    \put( 79,27.5){$\frac{\pi}{3}$}

    \put(-5,  1){$-1$}
    \put( 0, 77){$1$}
    \put( 0,116){$2$}
    \put( 0,154.5){$3$}
    \put(32,10){$\lambda_2$}
  \end{picture}
  \quad
  \begin{picture}(150,163)(-20,0)
    \put(0,0){\includegraphics{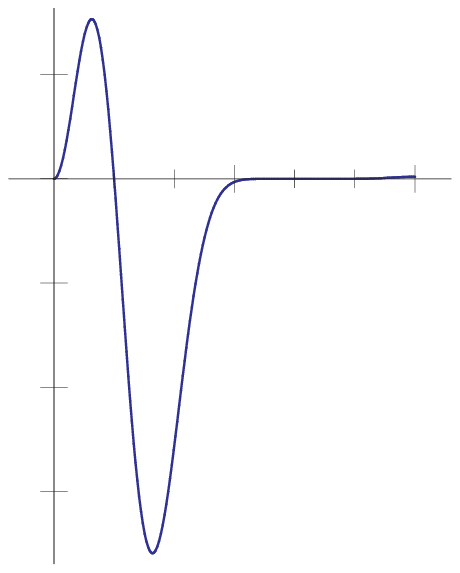}}
    \put(114,99){$\frac{\pi}{2}$}
    \put( 62,99){$\frac{\pi}{4}$}

    \put(-19, 19){$-150$}
    \put(-19, 49){$-100$}
    \put(-14, 79){$-50$}
    \put(-4,139){$50$}
    \put(65,10){$\lambda_2^8 F$}
  \end{picture}

\caption{The functions $F$, $\lambda_2$, and $\lambda_2^8F$.}
\label{F:graphs_a=2}
\end{figure}

We have
 \[
  \int_{0}^{\frac{\pi}{2}} \lambda_2^\mu F \ \geq \ 
  \int_{0}^{\frac{\pi}{3}} \lambda_2^\mu F\ - \ 
    \int_{\frac{\pi}{3}}^{\frac{\pi}{2}} \big| \lambda_2^\mu F\, \big| \,.
 \]
We prove Lemma~\ref{L:inequality_a=2} by showing that for $\mu\geq 14$, we have
\begin{equation}\label{Eq:main_inequality_a=2}
  \int_{0}^{\frac{\pi}{3}} \lambda_2^\mu F\ >\  
  \int_{\frac{\pi}{3}}^{\frac{\pi}{2}} \big| \lambda_2^\mu F\, \big| \,.
\end{equation}
We estimate the right-hand side.
On $[\frac{\pi}{3},\frac{\pi}{2}]$, the function $\lambda_2$ is 
decreasing and negative, so $| \lambda_2|\leq|\lambda_2(\frac{\pi}{2})| = 1$.
Similarly, the function $F$ increases from $-\frac{3}{2}$ at $\frac{\pi}{3}$ to $1$ at
$\frac{\pi}{2}$.
Thus 
\[ 
  \int_{\frac{\pi}{3}}^{\frac{\pi}{2}} \big| \lambda_2^\mu F \big| \ \leq \
  \int_{\frac{\pi}{3}}^{\frac{\pi}{2}} \frac{3}{2} \ = \ \frac{\pi}{4}\,.
\]
It is therefore enough to show that
\begin{equation}\label{Eq:RHS_main_inequality}
 \int_{0}^{\frac{\pi}{3}} \lambda_2^\mu F \ > \  
  \frac{\pi}{4}\,,
\end{equation}
for $\mu\geq 14$.
This inequality holds for $\mu=14$, as 
\[
  \int_{0}^{\frac{\pi}{3}} \lambda_2^{14} F \ =\ 
    \frac{1062882}{17017} \sqrt{3} +69 \pi\,.
\]

Suppose now that the inequality~\eqref{Eq:RHS_main_inequality} holds for some $\mu\geq 14$.
As $F$ is positive on $[0,\frac{\pi}{12}]$ and negative on
$[\frac{\pi}{12},\frac{\pi}{3}]$, this is equivalent to
\[
 \int_{0}^{\frac{\pi}{12}} \lambda_2^\mu F \ > \  
 -\int_{\frac{\pi}{12}}^{\frac{\pi}{3}} \lambda_2^\mu F\ + \ \frac{\pi}{4}\,, 
\]
and both integrals are positive.

For $\theta\in[0,\frac{\pi}{12}]$, $F(\theta)\geq 0$ and 
$\lambda_2(\theta)\geq \lambda_2(\frac{\pi}{12})=1+\sqrt{3}$ 
as $\lambda_2$ is decreasing on $[0,\frac{\pi}{2}]$.
Thus
 \begin{equation}\label{Eq:first_simplification}
   \int_{0}^{\frac{\pi}{12}}  \lambda_2^{\mu+1} F  \ \geq \ 
   \int_{0}^{\frac{\pi}{12}} \left(1{+}\sqrt{3}\right) \cdot \lambda_2^\mu F \,.
 \end{equation}
Similarly, for $\theta \in[\frac{\pi}{12},\frac{\pi}{3}]$, $F(\theta)\leq 0$ and
$1{+}\sqrt{3}\geq \lambda_2(\theta)\geq 0$, so 
\begin{equation}\label{Eq:second_simplification}
 -\int_{\frac{\pi}{12}}^{\frac{\pi}{3}} \left(1{+}\sqrt{3}\right) \cdot \lambda_2^\mu F
 \ \geq\ 
  -\int_{\frac{\pi}{12}}^{\frac{\pi}{3}} \lambda_2^\mu F\,.
\end{equation}

From the induction hypothesis and equations~\eqref{Eq:first_simplification}
and~\eqref{Eq:second_simplification}, we have 
 \begin{eqnarray*}
   \int_{0}^{\frac{\pi}{12}}  \lambda_2^{\mu+1} F &\geq& 
   \Bigl(1{+}\sqrt{3}\Bigr)  \cdot  \int_{0}^{\frac{\pi}{12}}\lambda_2^\mu F \\
  &>&  (1{+}\sqrt{3})\biggl( 
    -\int_{\frac{\pi}{12}}^{\frac{\pi}{3}} \lambda_2^\mu F\ +\ \frac{\pi}{4}\biggr) \\
  &>&-\int_{\frac{\pi}{12}}^{\frac{\pi}{3}}  \lambda_2^{\mu+1} F\ +\  \frac{\pi}{4}\,.    
 \end{eqnarray*}
This completes the proof of Lemma~\ref{L:inequality_a=2}. \hfill\qed

%
%
\subsection{Proof of Lemma~\ref{L:equal_inequality}}\label{SS:a=3}

We must show that $K(a^\mu,(a{-}1)^2) - K(a^\mu,2a)>0$ when $a\mu$ is even, $a\geq 3$, and 
$\mu\geq 2$.
We show the cases when $\mu=2,3$ by direct computation and then establish this inequality for
$\mu\geq 4$ by induction.

When $\mu=2$, we have $K(a^2,2a)=1$ and $K(a^2,(a{-}1)^2)=1+(a{-}1)=a$, by
Lemma~\ref{L:small_m}.
Thus $K(a^2,(a{-}1)^2)-K(a^2,2a)=a{-}1>0$ when $a\geq 3$.

When $\mu=3$, we must have that $a$ is even.  
Set $b:=a/2$.
The $K(a^3,2a)=1+b$ and $K(a^3(a{-}1)^2)=(5b^2+3b)/2$.
Then $K(a^3(a{-}1)^2)-K(a^3,2a)=\frac{1}{2}(5b^2+b-2)$, which is positive for
$b\geq 1$, and hence for $a\geq 2$.

By the integral formula for Kostka numbers~\eqref{Eq:Kostka_integral},
$K(a^\mu,(a{-}1)^2) - K(a^\mu,2a)$ is equal to 
 \begin{equation}\label{eq:integral_apositive}
   \frac{2}{\pi}\int_0^\pi \lambda_a(\theta)^\mu
    \bigl(\sin^2\, a\theta - \sin\,(2a{+}1)\theta\; \sin\,\theta\bigr)\, d\theta\ >\ 0\,.
 \end{equation}
Recall that $\lambda_a(\theta)=\frac{\sin(a{+}1)\theta}{\sin\theta}$ 
and write
\[
   \defcolor{F_a(\theta)}\ :=\  2(\sin^2\, a\theta - \sin\,(2a{+}1)\theta\; \sin\,\theta)
     \ =\ 1-2\cos 2a\theta+\cos\, (2a+2)\theta\,.
\]
These functions have symmetry about $\theta=\frac{\pi}{2}$,
\[
   F_a(\theta)\ =\ F_a(\pi-\theta)
    \qquad\mbox{\qquad}
   \lambda_a(\theta)\ =\ (-1)^a\lambda_a(\pi-\theta)\,.
\]
Thus if $a\mu$ is odd, the integral~\eqref{eq:integral_apositive} vanishes,
and it suffices to prove that
 \begin{equation}\label{Eq:whole_inequality}
   \int_0^\frac{\pi}{2} \lambda_a^\mu F_a\ >\ 0\,,
   \qquad \mbox{for all } a\geq 3\ \mbox{and}\ \mu\geq 4\,.
 \end{equation}
As in Subsection~\ref{SS:a=2}, we show this inequality by breaking the integral into two
pieces.
This is based on the following lemma, whose proof is given below.

\begin{lemma}\label{L:F_a}
  For $\theta\in[0,\frac{\pi}{a+1}]$, we have $\lambda_a(\theta)\geq 0$ and 
  $F_a(\theta)\geq 0$.
\end{lemma} 

 Thus we have,
 \[
   \int_0^\frac{\pi}{2} \lambda_a^\mu F_a\ >\ 
    \int_0^\frac{\pi}{a+1} \lambda_a^\mu F_a\ -\ 
    \int_{\frac{\pi}{a+1}}^\frac{\pi}{2} |\lambda_a^\mu F_a|\,,
 \]
and Lemma~$\ref{L:equal_inequality}$ follows from the following estimate.

\begin{lemma}\label{L:final_estimate}
 For every $a\geq 3$ and $\mu\geq 4$, we have 
 \begin{equation}\label{eq:integral_positive}
    \int_0^\frac{\pi}{a+1} \lambda_a^\mu F_a\ >\ 
    \int_{\frac{\pi}{a+1}}^\frac{\pi}{2} |\lambda_a^\mu F_a|\,.
 \end{equation}
\end{lemma}

We prove this inequality~\eqref{eq:integral_positive} by induction, 
first establishing the inductive step in Subsection~\ref{SS:ind} and 
then computing the base case in Subsection~\ref{SS:base}.

\begin{proof}[Proof of Lemma~$\ref{L:F_a}$]
 The statement for $\lambda_a$ is immediate from its definition.
 For $F_a$, we use elementary calculus.
 Recall that $F_a(\theta)=1-2\cos\, 2a\theta +\cos\, 2(a{+}1)\theta$, which equals
\[
    2(\sin^2\, a\theta - \sin\,(2a{+}1)\theta\; \sin\,\theta)\,.
\]
 Since the first term is everywhere nonnegative and the second nonnegative on
 $[\frac{\pi}{2a+1},\frac{2\pi}{2a+1}]$ (and $\frac{\pi}{a+1}<\frac{2\pi}{2a+1}$), we only
 need to show that $F_a$ is nonnegative on $[0,\frac{\pi}{2a+1}]$.
 Since $F_a(0)=0$, it will suffice to show that $F'_a$ is nonnegative on
 $[0,\frac{\pi}{2a+1}]$.

 As $F'_a=4a\sin\,2a\theta-2(a{+}1)\sin\,2(a{+}1)\theta$, we have $F_a'(0)=0$, and so it
 will suffice to show that $F''_a$ is nonnegative on  $[0,\frac{\pi}{2a+1}]$.
 Since $a>2$, we have $8a^2>4(a+1)^2$, and so
 \begin{eqnarray*}
   F_a''&=& 8a^2 \cos\, 2a\theta\ -\ 4(a{+}1)^2 \cos\, 2(a{+}1)\theta\\
        &&>\  4(a{+}1)^2(\cos\, 2a\theta\ -\ \cos\, 2(a{+}1)\theta)
        \ \:=\:\ 
        8(a{+}1)^2 \sin\,(2a{+}1)\theta\; \sin\,\theta\,.
 \end{eqnarray*}
 But this last expression is nonnegative on  $[0,\frac{\pi}{2a+1}]$.
\end{proof} 
 
Our proof of Lemma~\ref{L:final_estimate} will use the following well-known inequalities
for the sine function. 

\begin{proposition}\label{P:sine_ineqs}
 If $0\leq x\leq \frac{\pi}{2}$, then $\frac{2}{\pi}x\leq \sin\, x$.
 If $0\leq x\leq \frac{\pi}{4}$, then $\frac{2\sqrt{2}}{\pi}x\leq \sin\, x$.
 If $0\leq x\leq\pi$, then $\sin\,x\leq \frac{4}{\pi^2} x(\pi-x)$.
 Lastly, for every $x\geq 0$, we have
\begin{equation}\label{eq:MercerCaccia}
   3\frac{x}{\pi} - 4\frac{x^3}{\pi^3}\ \leq\ \sin\, x\ \leq\ x\,.
\end{equation}
\end{proposition}
The first two inequalities hold as the sine function is concave on the interval $[0,\frac{\pi}{2}]$,
and the last is standard.
The quadratic upper bound is derived in~\cite{QG}\footnote{For a (later) English version,
  see
  Xiaohui Zhang, Gendi Wang, and Yuming Chu, 
   {\it Extensions and Sharpenings of Jordan's and Kober's Inequalities}, JPIAM, {\bf 7}
   (2006), Issue 2, Article 63.}.
The cubic lower bound for sine is the \demph{Mercer--Caccia inequality}~\cite{Mercer-Caccia}.
We illustrate these bounds.
\[
  \begin{picture}(270,110)
    \put(0,0){\includegraphics{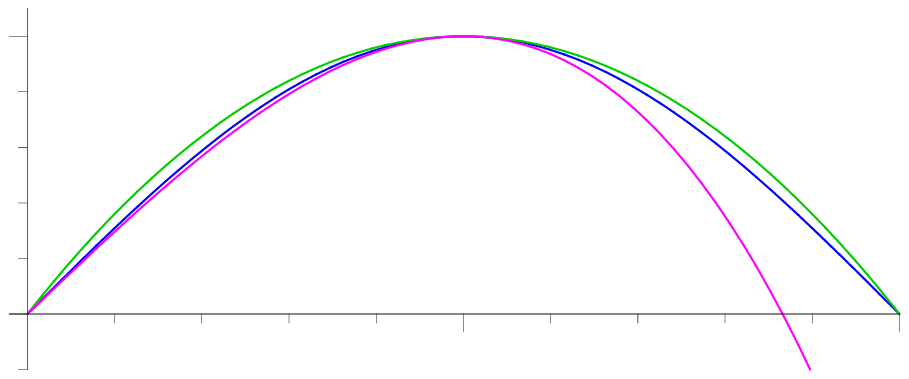}}
    \put(230,85){\Blue{$\sin\, x$}}
    \put(242.5,82.5){\Blue{\vector(-1,-1){20}}}
    \put(165,40){\Magenta{$3\frac{x}{\pi} - 4\frac{x^3}{\pi^3}$}}
    \put( 20,90){\Gre{$\frac{4}{\pi^2} x(\pi-x)$}}
    \put( 45,84){\Gre{\vector(1,-2){10}}}
    \put(  0,97){\small$1$}
    \put(  0,17){\small$0$}
    \put(136, 3){\small$\frac{\pi}{2}$}
    \put(263, 5){\small$\pi$}
  \end{picture}
\]

%
\subsubsection{Induction step of Lemma~$\ref{L:final_estimate}$}\label{SS:ind}
Our main tool is the following estimate.

\begin{lemma}\label{L:inductive_step}
  For all $a,\, \mu\geq 3$, we have
 \begin{equation}\label{E:inequality_for_inductivestep}
   \int_{0}^{\frac{\pi}{a+1}} \lambda_a^{\mu+1} F_a
   \ \:\geq\: \  \frac{(a{+}1)^3}{3(a{+}1)^2-4}\int_{0}^{\frac{\pi}{a+1}}\lambda_a^\mu F_a\,.
 \end{equation}
\end{lemma}

\begin{proof}[Induction step of Lemma~$\ref{L:final_estimate}$]

Suppose that we have 
 \begin{equation}\label{eq:positive_integral_difference}
    \int_0^{\frac{\pi}{a+1}} \lambda_a^\mu F_a\  > \ 
      \int_{\frac{\pi}{a+1}}^{\frac{\pi}{2}} \left| \ \lambda_a^\mu F_a \,  \right|\,,
 \end{equation}
for some number $\mu$.
We use the Mercer-Caccia inequality~\eqref{eq:MercerCaccia} at $x=\frac{\pi}{a+1}$ to obtain
\[
   \sin\tfrac{\pi}{a+1}\ \geq\ 3\frac{\frac{\pi}{a+1}}{\pi}\ -\
   4\frac{(\frac{\pi}{a+1})^3}{\pi^3}\ =\  
    \frac{3(a{+}1)^2-4}{(a{+}1)^3}\,.
\]
For $\theta\in[\frac{\pi}{a+1},\frac{\pi}{2}]$, we have
$\sin\theta\geq\sin\frac{\pi}{a+1}$ and $|\sin{(a{+}1)\theta}|\leq 1$, and therefore
 \begin{equation}\label{Eq:Ca}
  |\lambda_a(\theta)|\ =\ \left|\frac{\sin{(a{+}1)\theta}}{\sin{\theta}}\right| \leq 
  \left|\frac{1}{\sin\frac{\pi}{a+1}}\right|\ \leq\ \frac{(a+1)^3}{3(a+1)^2 - 4}\,.
 \end{equation}
This last number is the constant in Lemma~\ref{L:inductive_step}, which we now denote by 
\defcolor{$C_a$}.
By Lemma~\ref{L:inductive_step}, our induction
hypothesis~\eqref{eq:positive_integral_difference}, and~\eqref{Eq:Ca}, we have 
\[
  \int_{0}^{\frac{\pi}{a+1}} \lambda_a^{\mu+1} F_a
  \ \geq\ C_a \int_{0}^{\frac{\pi}{a+1}} \lambda_a^\mu F_a
  \ \geq\ C_a \int_{\frac{\pi}{a+1}}^{\frac{\pi}{2}} \big|\,\lambda_a^\mu F_a \, \big|
  \ \geq\  \int_{\frac{\pi}{a+1}}^{\frac{\pi}{2}} \big|\,\lambda_a^{\mu+1} F_a \, \big|\,,
\]
which completes the induction step of Lemma~\ref{L:final_estimate}.
\end{proof}

Our proof of Lemma~\ref{L:inductive_step} uses some linear bounds for $\lambda_a$.
To gain an idea of the task at hand, in Figure~\ref{F:integral} we show the integrand
$\lambda_a^\mu F_a$ and   
$\lambda_a$ on $[0,\frac{\pi}{a+1}]$, for $a=4$ and $\mu=2$.  
\begin{figure}[htb]
 \begin{picture}(185,239)(-2,-2)
  \put(2,0){\includegraphics[height=236pt]{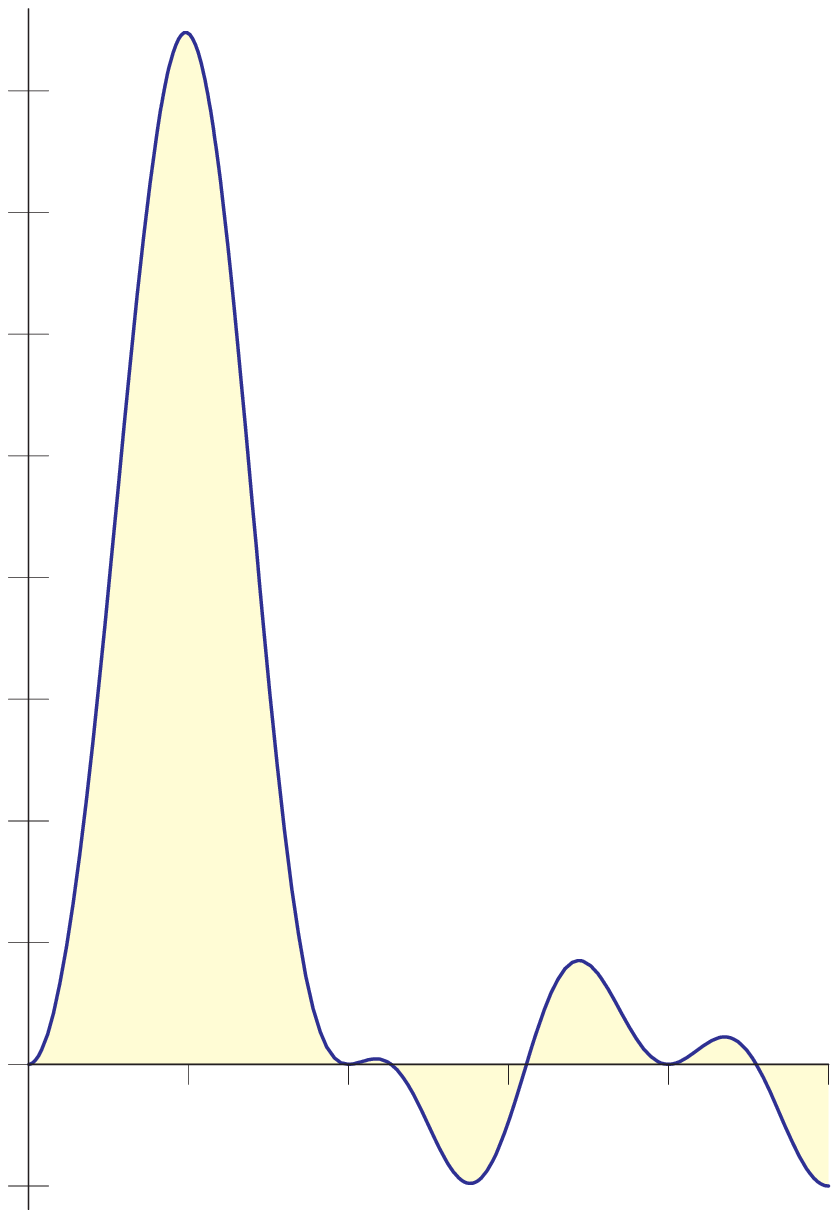}}
  \put(6,217){\small$8$}
  \put(6,169){\small$6$}
  \put(6,121){\small$4$}
  \put(6, 73){\small$2$}
  \put(-2,2){\small $-1$}
  \put(63,130){$\lambda_4^2 F_4$}
  \put( 75,13){\small$\frac{\pi}{5}$}
  \put(137,13){\small$\frac{2\pi}{5}$}
  \put(170,13){\small$\frac{\pi}{2}$}
 \end{picture}
   \qquad
  \begin{picture}(160,239)(-3,-1)
  \put(0,0){\includegraphics[height=236pt]{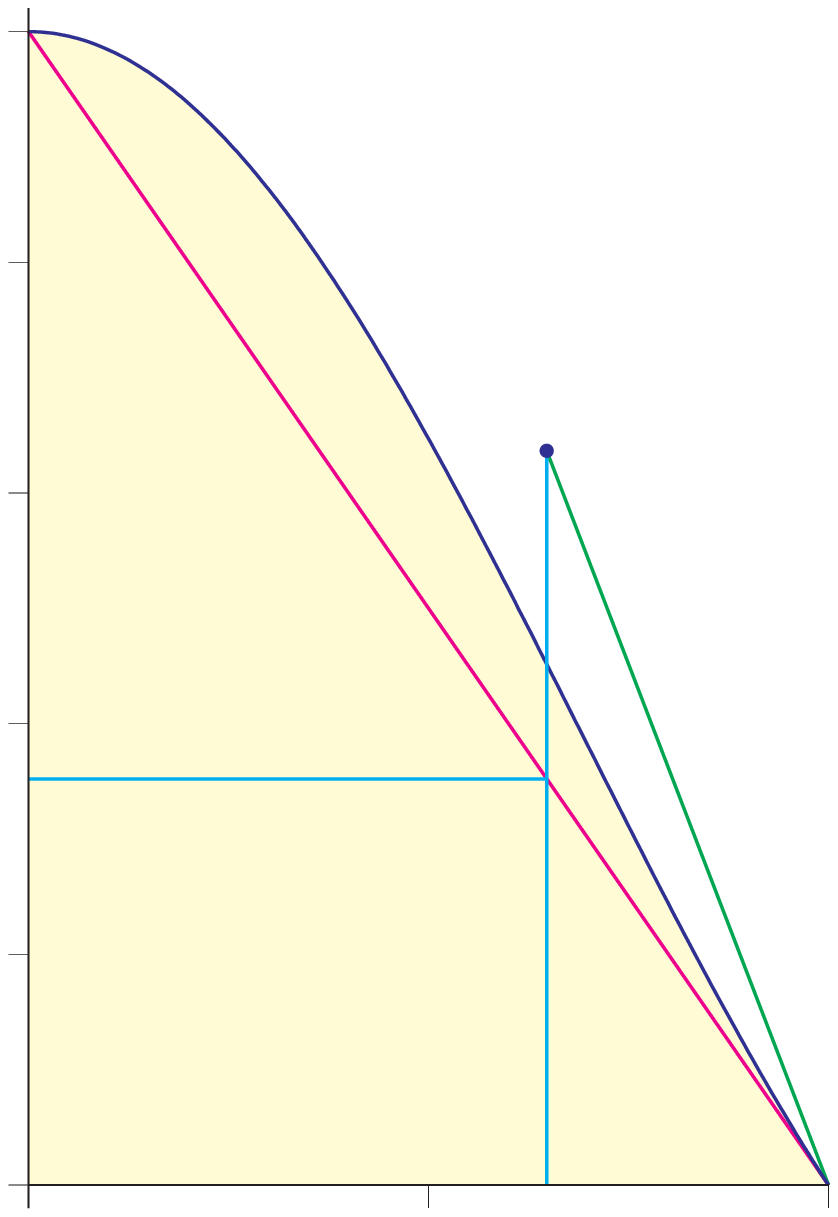}}
  \put(-3,228){\small$5$}
  \put(-3,187){\small$4$}
  \put(-3,145){\small$3$}
  \put(-3,103){\small$2$}
  \put(-3, 61){\small$1$}
  \put(-3, 19){\small$0$}
  \put(57,200){$\lambda_4$}
  \put( 77,8){\small$\frac{\pi}{10}$}
  \put(149,8){\small$\frac{\pi}{5}$}

  \put( 99,10){\small$b$}
  \put(-8,92){$C_a$}
  \put(25,180){$\ell_a$}
  \put(85,163){$(b,\frac{2\pi(a{+}1)}{\pi}$)}
  \put(123,110){$\calL_a$}  
 \end{picture}
  \caption{The integrand $\lambda_4^2 F_4$ and $\lambda_4$.}
  \label{F:integral}
\end{figure} 

We estimate $\lambda_a$.
Define the linear function
\[
   \defcolor{\ell_a(\theta)}\ :=\ \tfrac{(a{+}1)^2}{\pi}(\tfrac{\pi}{a+1} - \theta)\,,
\]
which is the line through the points $(0,a{+}1)$ and $(\frac{\pi}{a+1},0)$ on the graph of
$\lambda_a$.

\begin{lemma}\label{L:ell_a}
 For $\theta$ in the interval $[0,\frac{\pi}{a+1}]$, we have 
 $\ell_a(\theta)\leq \lambda_a(\theta)$.
\end{lemma}

\begin{proof}
 We need some information about the derivatives of $\lambda_a(\theta)$.
 First observe that
 \begin{eqnarray*}
  \lambda_a(\theta)&=&\frac{\sin(a{+}1)\theta}{\sin\theta}\ =\ 
   \frac{e^{i(a+1)\theta}-e^{-i(a+1)\theta}}{e^{i\theta}-e^{-i\theta}}\ =\ 
   \sum_{j=0}^a e^{i(a-2j)\theta}\\
  &=& 2\cos a\theta\ +\ 2\cos(a{-}2)\theta \ +\ \dotsb\ +\ 
    \left\{ \begin{array}{ccl} 2\cos\theta&&\mbox{if $a$ is odd}\\ 
                               1&&\mbox{if $a$ is even}\end{array}\right.
 \end{eqnarray*}
 From this, we see that $\lambda_a'(0)=0$ and 
 $\lambda_a'$ is negative on $(0,\frac{\pi}{a+1})$.
 Moreover, $\lambda_a''$ is a sum of terms of the form
 $-2(a{-}2j)^2\cos(a{-}2j)\theta$, for $0\leq j<\frac{a}{2}$.
 Thus $\lambda_a''$  is increasing on $[0,\frac{\pi}{a+1}]$, as each term
 is increasing on that interval.

 Since $\ell_a$ has negative slope and $\lambda_a'(0)=0$, we have 
 $\ell_a(\theta)< \lambda_a(\theta)$ for $\theta\in [0,\frac{\pi}{a+1}]$ near $0$.
  We compute $\lambda_a'(\frac{\pi}{a+1})$.
  Since
\[
   \lambda_a'(\theta)\ =\ \frac{(a{+}1)\cos(a{+}1)\theta}{\sin\theta}\ -\ 
   \frac{\cos\theta\sin(a{+}1)\theta}{\sin^2\theta}\,,
\]
 we have 
\[
   \lambda_a'(\tfrac{\pi}{a+1})\ =\ -\frac{a{+}1}{\sin\frac{\pi}{a+1}}\ <\ -\frac{(a{+}1)^2}{\pi}\,,
\]
 as $0<\sin\frac{\pi}{a+1}<\frac{\pi}{a+1}$.
 Thus at $\theta=\frac{\pi}{a+1}$, we have $\lambda_a(\theta)=\ell_a(\theta)=0$ and
 $\lambda_a'(\theta) < \ell'_a(\theta)$ and so $\ell_a(\theta)< \lambda_a(\theta)$  for
 $\theta\in [0,\frac{\pi}{a+1}]$ near $\frac{\pi}{a+1}$.

 If $\ell_a(\theta)> \lambda_a(\theta)$ at some point $\theta\in(0,\frac{\pi}{a+1})$, 
 then we would have $\ell_a(\theta)= \lambda_a(\theta)$ for at least two points 
 $\theta$ in $(0,\frac{\pi}{a+1})$.
 Since $\ell_a(\theta)= \lambda_a(\theta)$ at the endpoints, Rolle's Theorem would imply that
 $\lambda_a''$ has at least two zeroes in $(0,\frac{\pi}{a+1})$, which is impossible as
 $\lambda_a''$ is increasing. 
\end{proof}

\begin{proof}[Proof of Lemma~$\ref{L:inductive_step}$]
 By Lemma~\ref{L:ell_a}, we have 
 \[
   \int_{0}^{\frac{\pi}{a+1}} \lambda_a^{\mu+1} F_a
   \ \geq \ \int_{0}^{\frac{\pi}{a+1}} \ell_a \lambda_a^{\mu} F_a\,,
 \]
 and so it suffices to prove
 \[
   \int_{0}^{\frac{\pi}{a+1}} \ell_a\lambda_a^{\mu} F_a
   \ \geq \ C_a\,\int_{0}^{\frac{\pi}{a+1}} \lambda_a^{\mu}F_a\,.
 \]
This is equivalent to showing that
 \begin{equation}\label{Eq:ineq_line}
   \int_{0}^{\frac{\pi}{a+1}} (\ell_a-C_a)\lambda_a^{\mu} F_a
  \ \geq \ 0.
 \end{equation}
As $\defcolor{L_a}:=\ell_a - C_a$ is linear, this is the difference of two integrals of positive
functions.
We establish the inequality~\eqref{Eq:ineq_line} by estimating each of those integrals.

The function $L_a$ is a line with slope $-\frac{(a+1)^2}{\pi}$ and zero at
 \[
   \defcolor{b}\ :=\  \frac{2(a^2+2a-1)\pi}{(a+1)(3a^2+6a-1)}\ \in\ 
    \left[\frac{\pi}{2(a{+}1)}\,,\,\frac{\pi}{a{+}1}\right]\,.  
 \]
The inequality \eqref{Eq:ineq_line} is equivalent to 
 \begin{equation}\label{Eq:another_one}
  \int_{0}^{b} L_a\, \lambda_a^\mu F_a\ \geq\ 
  \int_{b}^{\frac{\pi}{a+1}} | L_a | \, \lambda_a^\mu F_a\,.
 \end{equation}

For $\theta\in[0,\frac{\pi}{2(a{+}1)}]$, the linear inequalities of Proposition~\ref{P:sine_ineqs}
give 
 \[
  \sin{(a{+}1)\theta}\ \geq\ \frac{2}{\pi}(a{+}1)\theta
    \qquad\mbox{and}\qquad
    \sin \theta\ \leq\ \theta\,,
 \]
 and thus 
\[
    \lambda_a(\theta)\ =\ \frac{\sin{(a{+}1)\theta}}{\sin\theta}\ \geq\
     \frac{2(a{+}1)}{\pi}\,.
\]
 Since $L_a\lambda_a^\mu F_a$ is nonnegative on $[0,b]$ and $\frac{\pi}{2(a{+}1)}<b$,
 we have 
\[
   \int_{0}^{b} L_a\, \lambda_a^\mu F_a  \ \geq\ 
   \int_{0}^{\frac{\pi}{2(a+1)}} L_a\, \lambda_a^\mu F_a\ \geq\ 
   \frac{2^\mu (a{+}1)^\mu}{\pi^\mu} \int_{0}^{\frac{\pi}{2(a+1)}} L_a\, F_a\,.
\]
 We may exactly compute this last integral to obtain
%
%
%
%
%
%
 \begin{eqnarray*}
   \int_{0}^{\frac{\pi}{2(a+1)}} L_a\, F_a &=& \frac{1}{8\pi a^2 (3a^2+6a -1)}\cdot [
   (5\pi^2a^4 + (10\pi^2{-}24)a^3 - (7\pi^2{+}60)a^2 - 16a +4) \\
   &&+\cos{\frac{a\pi}{a+1}} \cdot (12a^4 + 48a^3 + 56a^2 +16a - 4)\\
   &&+\sin{\frac{a\pi}{a+1}} \cdot (-4\pi a^4 - 12\pi a^3 + 4\pi a^2 +12\pi a)]\,.
 \end{eqnarray*}
As $a>1$, we have $\cos{\frac{a\pi}{a+1}} > -1$ and $\sin{\frac{a\pi}{a+1}}> 0$.
Substituting these values into this last formula and multiplying by $(2(a{+}1)/\pi)^\mu$
gives a lower bound for the integral on the left of~\eqref{Eq:another_one},
 \begin{equation}
  \defcolor{A}\ :=\ 
   \frac{2^\mu (a+1)^\mu ((5\pi^2{-}12)a^4 + (10\pi^2{-}72)a^3 - (7\pi^2{+}116)a^2 - 32a +8)}%
    { 8 \pi^{\mu+1} a^2 (3a^2+6a -1)}\,.
 \end{equation}
%
%

For the integral on the right of~\eqref{Eq:another_one}, consider the line through the
points $(\frac{\pi}{a+1},0)$ and $(b,\frac{2(a{+}1)}{\pi})$,  
%
%
 \[
   \defcolor{\calL_a}\ :=\ \frac{2(3a^2+6a-1)}{\pi^2}\left( \frac{\pi}{a{+}1}-\theta \right)\,.
 \]
We claim that $\lambda_a < \calL_a$ in the interval $[b,\frac{\pi}{a+1}]$.
To see this, first note that the slope of a secant line through $(\frac{\pi}{a+1},0)$ and a
point $(\theta,\lambda_a(\theta))$ on the graph of $\lambda_a$ is 
 \begin{equation}\label{Eq:slope}
   \frac{\sin{(a{+}1)\theta}}{(\theta-\frac{\pi}{a+1})\sin{\theta}}\,.
 \end{equation}
As observed in Proposition~\ref{P:sine_ineqs}, $\sin{(a{+}1)\theta}$ is bounded above by the parabola,
 \[
   \sin{(a{+}1)\theta}\ \leq\
     \frac{4(a{+}1)^2}{\pi^2} \theta\, \left(\frac{\pi}{a{+}1} - \theta\right)\,.
 \]
We use this and the Mercer--Caccia inequality~\eqref{eq:MercerCaccia} for $\sin{\theta}$ to
bound the slope~\eqref{Eq:slope},
 \[ 
   \frac{\sin{(a{+}1)\theta}}{(\theta-\frac{\pi}{a{+}1})\sin{\theta}}
   \  \leq\ 
    \frac{4 \pi (a{+}1)^2}{(3\pi^2-4\theta^2)}
   \ \leq\ \frac{4(a{+}1)^4}{\pi(3a^2+6a-1)},
 \]
with the second equality holding as the minimum of the denominator $(3\pi^2-4\theta^2)$ on the
interval $[b,\frac{\pi}{a+1}]$ occurs at $\theta=\frac{\pi}{a+1}$. 
When $a\geq 3$ we have,
 \[
   \frac{4(a+1)^4}{\pi(3a^2+6a-1)}\ <\ \frac{2(3a^2+6a-1)}{\pi^2},
 \]
which so it follows that $\lambda_a< \calL_a$ on $[b,\frac{\pi}{a+1}]$. 

Using this and the easy inequality $F_a< 4$, we bound the integral on the right of~\eqref{Eq:another_one},
 \[ 
   \int_{b}^{\frac{\pi}{a+1}} |L_a| \, \lambda_a^\mu F_a
    \ <\
   \int_{b}^{\frac{\pi}{a+1}} |L_a|\, \calL^\mu F_a
    \ <\ 
    \int_{b}^{\frac{\pi}{a+1}} 4 |L_a|\, \calL^\mu\,.
 \]
  The last integral is not hard to compute,
\[ 
    \defcolor{B}\ :=\ \int_{b}^{\frac{\pi}{a+1}} 4 |L_a|\, \calL_a^\mu
    \ =\ \frac{2^{\mu+2}(a+1)^{\mu+3}[\mu+1-(a+1)(\mu+2)]}{\pi^{\mu-1}(\mu+1)(\mu+2) (3a^2+6a-1)^2}\;.
\]
 We claim that $A-B > 0$, which will complete the proof of Lemma~\ref{L:inductive_step} and
 therefore the induction step for Lemma~\ref{L:final_estimate}.
 For this, we observe that if multiply $A-B$ by their common (positive) denominator,
 we obtain an expression of the form $2^\mu(a+1)^\mu P(a,\mu)$, where $P$ is a polynomial of
 degree six in $a$ and two in $\mu$.
 After making the substitution $P(3+x,3+y)$, we obtain a polynomial in $x$ and $y$ in which
 every coefficient in positive, which implies that $A-B>0$ when $a,m\geq 3$, and completes the
 proof. 
\end{proof}

%
\subsubsection{Base of the induction for Lemma~$\ref{L:final_estimate}$}\label{SS:base}

We establish the inequality~\eqref{eq:integral_positive} of Lemma~\ref{L:final_estimate} when
$\mu=4$, which is the base case of our inductive proof.
This inequality is
 \begin{equation}\label{Eq:mu=4}
    \int_0^{\frac{\pi}{a+1}} \lambda_a^4\, F_a\ >\ 
    \int_{\frac{\pi}{a+1}}^{\frac{\pi}{2}} |\lambda_a^4\, F_a|
   \qquad\mbox{for every }a\geq 3\,.
 \end{equation}
We establish this inequality by replacing each integral by one which we
may evaluate in elementary terms, and then compare the values.

We first find an upper bound for the integral on the right.
Recall that
\[
   \lambda_a(\theta)\ =\ \frac{\sin(a{+}1)\theta}{\sin \theta}
    \qquad\mbox{and}\qquad
   F_a(\theta)\ =\ 1-2\cos 2a\theta + \cos 2(a{+}1)\theta\,.
\]
Since $|\lambda_a(\theta)|\leq\frac{1}{\sin\theta}$ and $|F_a(\theta)|\leq 4$ for 
$\theta\in[\frac{\pi}{a+1},\frac{\pi}{2}]$, we have
\[
   \int_{\frac{\pi}{a+1}}^{\frac{\pi}{2}} |\lambda_a^4\, F_a|
     \ \leq\ 
   4\int_{\frac{\pi}{a+1}}^{\frac{\pi}{2}} \frac{1}{\sin^4\theta}
    \ =\ 
   \frac{4}{3}\cot \tfrac{\pi}{a{+}1}\bigl( 2 + \csc^2\tfrac{\pi}{a{+}1})\,.
\]
For $a\geq 3$, we have $0<\frac{\pi}{a+1}\leq\frac{\pi}{4}$.
As we observed in Proposition~\ref{P:sine_ineqs}, this implies that 
$\sin\frac{\pi}{a+1}\geq\frac{\pi}{a+1}\frac{2\sqrt{2}}{\pi}=\frac{2\sqrt{2}}{a+1}$,
and so $\frac{1}{\sin\frac{\pi}{a+1}}\geq \frac{a+1}{2\sqrt{2}}$.
Since $0\leq\cos\frac{\pi}{a+1}\leq 1$, we have
 \begin{equation}\label{eq:rhs_estimate}
   \frac{4}{3}\cot \tfrac{\pi}{a{+}1}\bigl( 2 + \csc^2\tfrac{\pi}{a{+}1})
    \ \leq\  \frac{4(a{+}1)}{3\sqrt{2}} + \frac{(a{+}1)^3}{12\sqrt{2}}
   \ =:\ \defcolor{B}\,.
 \end{equation}

We now find a lower bound for the integral on the left of~\eqref{Eq:mu=4}.
We use the estimate from Lemma~\ref{L:ell_a}, that for $\theta\in[0,\frac{\pi}{a+1}]$, we have
\[
   \lambda_a(\theta)\ \geq\ \ell_a(\theta)\ =\ 
    \frac{(a+1)^2}{\pi}\left( \frac{\pi}{a+1}-\theta\right)\,.
\]
Using this gives the lower bound,
\[
  \int_0^{\frac{\pi}{a+1}} \lambda_a^4 F_a\ >\ 
  \frac{(a+1)^8}{\pi^4} \int_0^{\frac{\pi}{a+1}}\left( \tfrac{\pi}{a+1}-\theta\right)^4 
     \bigl(1-2\cos 2a\theta+\cos 2(a{+}1)\theta\bigr)\,.
\]
This may be evaluated in elementary terms to obtain
 \begin{equation}\label{Eq:Exact_estimate}
  \frac{3(a{+}1)^8}{2a^5\pi^4}\sin\tfrac{2\pi}{a+1}
   +\frac{\pi(a+1)^3}{5}-\frac{2(a+1)^5}{\pi a^2}
   +\frac{3(a+1)^7}{\pi^3 a^4} + \frac{(a+1)^3}{\pi}
   -\frac{3(a+1)^3}{2\pi^3}\ .
 \end{equation}
For $a\geq 3$, $0\leq \frac{2\pi}{a+1}\leq\frac{\pi}{2}$, we have the bound from
Proposition~\ref{P:sine_ineqs} of $\sin\frac{2\pi}{a+1}\geq\frac{4}{a+1}$.
Thus the expression~\eqref{Eq:Exact_estimate} is bounded below by
 \begin{equation}\label{Eq:rational_estimate}
   \defcolor{A}\ :=\ \frac{6(a{+}1)^7}{\pi a^5}   +\frac{\pi(a+1)^3}{5}-\frac{2(a+1)^5}{\pi a^2}
   +\frac{3(a+1)^7}{\pi^3 a^4} + \frac{(a+1)^3}{\pi}
   -\frac{3(a+1)^3}{2\pi^3}\ .
 \end{equation}
Then the difference $A-B$ of the expressions from~\eqref{Eq:rational_estimate}
and~\eqref{eq:rhs_estimate} is a rational function of the form
\[
    \frac{(a+1)\cdot P(a)}{120 \pi^4 a^5}\ ,
\]
where $P(a)$ is a polynomial of degree seven.
If we expand $P(3+x)$ in powers of $x$, then we obtain a polynomial of degree seven in $x$
with positive coefficients.
This establishes the inequality~\eqref{Eq:mu=4} for all $a\geq 3$, which is the base case
of the induction proving Lemma~\ref{L:final_estimate}.
This completes the proofs of Lemma~\ref{L:final_estimate},
Lemma~\ref{L:equal_inequality}, and ultimately of Theorem~\ref{Th:one}. \qed

\providecommand{\bysame}{\leavevmode\hbox to3em{\hrulefill}\thinspace}
\providecommand{\MR}{\relax\ifhmode\unskip\space\fi MR }
\providecommand{\MRhref}[2]{%
  \href{http://www.ams.org/mathscinet-getitem?mr=#1}{#2}
}
\providecommand{\href}[2]{#2}


\end{document}